\documentclass[12pt]{article}
\usepackage{amsmath}
\usepackage{amsthm}
\usepackage{amssymb}

\usepackage{hyperref}
\usepackage{enumerate}

\newcommand{\defn}[1]{{\em #1}\index{#1}}

\usepackage{color}
\newcommand\comment[1]{\hspace{-8px}{\color{purple}{\bf *}}}

\usepackage{tikz}

\theoremstyle{definition} 

\newtheorem{thm}{Theorem}[section]
\newtheorem{definition}[thm]{Definition}
\newtheorem{lem}[thm]{Lemma}
\newtheorem{cor}[thm]{Corollary}

\newtheorem{prop}[thm]{Proposition}
\newtheorem{remark}[thm]{Remark}

\newcommand\norm[1]{\| #1 \|}
\newcommand\abs[1]{\left\lvert #1 \right\rvert}
\newcommand\pmat[1]{\begin{pmatrix}#1\end{pmatrix}} 
\newcommand\smat[1]{\left(\begin{smallmatrix}#1\end{smallmatrix}\right)} 
\newcommand\isom{\cong}

\def\<{\langle}
\def\>{\rangle} 

\newcommand\floor[1]{\lfloor #1 \rfloor}
\newcommand\ceil[1]{\lceil #1 \rceil}

\DeclareMathOperator\id{id}
\DeclareMathOperator\im{im}
\DeclareMathOperator\ind{ind}

\DeclareMathOperator\supp{supp}
\DeclareMathOperator\Span{span}
\DeclareMathOperator\aut{aut}

\def\bR{\mathbb{R}}
\def\bN{\mathbb{N}}
\def\bZ{\mathbb{Z}}
\def\bI{\mathbb{I}}
\def\cA{\mathcal{A}}

\def\cC{\mathcal{C}}

\def\cL{\mathcal{L}}
\def\cM{\mathcal{M}}

\def\cP{\mathcal{P}}
\def\cR{\mathcal{R}}
\def\cS{\mathcal{S}}
\def\cT{\mathcal{T}}
\def\cV{\mathcal{V}}
\def\cZ{\mathcal{Z}}

\def\brho{{\bar\rho}}

\def\blambda{{\bar\lambda}}
\def\ab{{\alpha,\beta}}

\def\bcP{{\bar\cP}}
\def\bcR{{\bar\cR}}

\newcommand{\ft}{{\mathfrak t}}

\newcommand{\case}[1] { \vspace{.1in} \noindent {\em Case #1\/}}
\newcommand{\step}[1] { \vspace{.1in} \noindent {\em Step #1\/}}

\begin{document}
\title{Combinatorial embedded contact homology for toric contact manifolds}
\author{Keon Choi\footnote{Supported by ERC advanced grant LDTBud}}
\date{}
\maketitle

\begin{abstract}
Computing embedded contact homology (ECH) and related invariants of certain toric 3-manifolds (in the sense of Lerman \cite{lerman}) has led to interesting new results in the study of symplectic embeddings \cite{cfghr,concconv,qech}. Here, we give a combinatorial formulation of ECH chain complexes for general toric contact 3-manifolds. As a corollary, we prove Conjecture A.3 from \cite{beyond}.
\end{abstract}

\tableofcontents

\section{Introduction}\label{section:intro}
Embedded contact homology (ECH) is an invariant of a contact 3-manifold. The goal of this paper is to combinatorially describe ECH chain complexes (ECC) of certain contact manifolds. A combinatorial formulation of Heegaard Floer homology (which is isomorphic to ECH by \cite{cgh2,klt}) is given in \cite{sw} but understanding ECC itself is useful for studying contact geometric properties lost under this isomorphism: e.g. usage of ECH capacities \cite{cfghr,concconv,qech} and other obstructions \cite{beyond} to symplectic embeddings.

In \cite{t3}, Hutchings and Sullivan introduced ``polygonal paths'' and ``rounding corners'' to describe the generators and differentials of ECC for $(T^3,\lambda_n)$ where $\lambda_n:=\cos (2\pi nx) dt_1+ \sin (2\pi nx) dt_1$. We extend this result to all toric contact 3-manifolds, that is, $(Y^3,\lambda)$ with a $\lambda$-preserving effective $T^2$-action. According to Lerman \cite{lerman}, such $Y$ admits the contact moment map $\mu_\lambda:Y \to (\ft^2)^*$ which factors through its orbital moment map $a_\lambda: Y/T^2 \to (\ft^2)^*$. If the action is free, $Y/T^2$ is homeomorphic to $\bR/\bZ$ and $Y$ is diffeomorphic to $T^3$. Otherwise, $Y/T^2$ is homeomorphic to $[0,1]$ and $Y$ is diffeomorphic to a lens space (including $S^1 \times S^2$). In either case, $Y$ contains as a dense open submanifold a principal $T^2$-bundle $Y^o :=(0,1) \times T^2$ and after re-identifying the fibres if necessary, $\lambda|_{Y^o} = \pi^*(a_\lambda|_{(0,1)})$ where $\pi:Y\to Y/T^2$.

Let $\bI=[0,1]$ and $Y=\bI \times T^2$ with the projection $\pi_\bI:Y \to \bI$. For any $a:\bI \to (\ft^2)^*=T^*T^2$, $\pi_\bI^*a:Y \to (\ft^2)^* \subset T^*Y$, considered as a 1-form, is contact if and only if $a \times a' > 0$ where $\times$ is the standard cross product on $(\ft^2)^* = \bR^2$. We call any such $a$ an (abstract) orbital moment map. Our main theorem describes $ECC(Y,\lambda,J)$ where $\lambda$ is a certain perturbation of $\pi_\bI^*a$ for a generic orbital moment map $a$ and $J$ is a certain generic $\lambda$-admissible almost complex structure on $\bR \times Y$. Recalling ECC is generated over $\bZ/2$ by admissible orbit sets of $\lambda$ and the differential $\partial$ counts ECH index 1 $J$-holomorphic curves (see \S \ref{section:ech}), we show:
\begin{thm}\label{thm:main}
Let $(Y,\lambda,J)$ be as above. For a pair $(\alpha,\beta)$ of admissible orbit sets of $\lambda$, $\< \partial \alpha, \beta\>  =1 \in \bZ/2$ if and only if the region $\cR_\ab$ associated to it can be written as a concatenation $\cT_1\cR' \cT_2$ where $\cT_i$ are trivial regions and $\cR'$ is a non-local, indecomposable, $a$-positive, minimally positive and almost minimally decorated region.
\end{thm}

Before giving precise definitions, it is useful to have in mind:
\begin{enumerate}[(a)]
\item The (rough) correspondence between: a trivial region and a trivial cylinder; an indecomposable region and an irreducible $J$-holomorphic curve; a concatenation of regions and a disjoint union of $J$-holomorphic curves; and $a$-positivity and intersection positivity.
\item The combinatorial ECH index of a non-local, indecomposable, $a$-positive region is non-negative. It is zero if and only if the region is minimally positive and minimally decorated.
\end{enumerate}

Since the Reeb vector field $\bar R$ of $\pi_\bI^*a$ at $(x,t_1,t_2) \in Y$ takes values in $\ker a'(x) \subset \ft^2 = T^{vert}_{(x,t_1,t_2)}Y$, an $S^1$-family $\brho_x$ of embedded orbits foliates $\{x\} \times T^2$ whenever $a'(x)$ is a multiple of an integral vector. The set of such $x$ is dense generically but a technical argument allows us to only consider $\brho_x$ containing orbits whose action is less than a fixed constant $L$ (See \S \ref{section:ech}.) Then, following Bourgeois \cite{mb}, we perturb $\pi_\bI^*a$ to $\lambda$, which has exactly two embedded orbits with action less than $L$ (one elliptic orbit $e_x$ and one positive hyperbolic orbit $h_x$) for each such $\brho_x$ (see \S \ref{section:mb}). In addition to $[e_x]=[h_x] \in H_1(T^2)$, $\brho_x$ has another important attribute:
\begin{definition}\label{def:convex}
We say $a$ is \defn{convex} at $x$ (or $\brho_x$ is \defn{convex}) if $a'(x) \times a''(x) >0$; $a$ is \defn{concave} at $x$ (or $\brho_x$ is \defn{concave}) if $a'(x) \times a''(x) <0$.
\end{definition}

In $(T^3,\lambda_n)$, $a_{\lambda_n}(x)= (\cos 2\pi nx,\sin 2\pi nx)$ so every $\brho_x$ is convex whereas in $(S^3,\lambda_{std})$, $a_{\lambda_{std}}(x) = (1-x,x)$ so no $\brho_x$ is convex or concave. By genericity of $a$, we assume every orbit of action less than $L$ is either convex or concave.
\begin{definition}\label{def:path}
Let $\Lambda\subset \ft^2$ be the kernel of the exponential map (hence, naturally identified with $H_1(T^2)$). A \defn{(lattice) path} $\bcP$ is a function
\[ (v_\bcP,c_\bcP,m_\bcP):\bI \to \bar\cV := \Lambda \times \{\pm 1,0\} \times \bN \]
such that $v_\bcP,c_\bcP$ and $m_\bcP$ (read \defn{edge}, \defn{convexity} and \defn{multiplicity}) are non-zero on the same finite set $\supp\bcP$ (read the \defn{support} of $\bcP$) and $v_\bcP(x)$ is primitive whenever non-zero. We write $m(\bcP) = \sum_x m_\bcP(x)$ and $[\bcP]:=\sum_x m_\bcP(x) \cdot v_\bcP(x)$ and say:
\begin{enumerate}[(a)]
\item Two paths $\bcP_1$ and $\bcP_2$ are \defn{compatible} if  $v_{\bcP_1}(x)=v_{\bcP_2}(x)$ and $c_{\bcP_1}(x)= c_{\bcP_2}(x)$ for every $x \in \supp(\bcP_1) \cap \supp(\bcP_2)$. In this case, their \defn{union} $\bcP_1 \cup \bcP_2$ is the path with $m_{\bcP_1 \cup\bcP_2} = m_{\bcP_1} + m_{\bcP_2}$ and compatible with each $\bcP_i$.
\end{enumerate}
A \defn{decoration} $\cP$ of $\bcP$ is a function $(v_\cP,c_\cP,m_\cP^e,m_\cP^h):\bI \to \cV := \Lambda \times \{\pm 1,0\} \times \bN^2$ ($m^e_\cP$ and $m^h_\cP$ read \defn{elliptic} and \defn{hyperbolic multiplicity}) with $\bcP=(v_\cP,c_\cP,m_\cP^e+m_\cP^h)$.
\end{definition}

We use the term ``path'' because we can depict $\bcP$ as a piecewise linear curve in $\ft^2$ by concatenating $v_\bcP(x)$ with multiplicity $m_\bcP(x)$ in order of increasing $x$, where each instance of $v_\bcP(x)$ is labelled with $\check x$ if $c_\bcP(x)=1$ (\defn{convex}) and with $\hat x$ if $c_\bcP(x)=-1$ (\defn{concave}). For a decorated path $\cP$, we label each edge with $\check e_x, \check h_x, \hat h_x$ or $\hat e_x$ (this is unique only up to shuffling $e/h$ labels at the same $x$) as we see shortly in Figure \ref{fig:diff-examples}.

\begin{definition}\label{def:region}
A \defn{(lattice) region} $\bcR$ is a pair $(\bcP^0,\bcP^1)$ of compatible lattice paths with $[\bcP^0] = [\bcP^1]$. We write $c_\bcR= c_{\bcP^0 \cup \bcP^1}, v_\bcR= v_{\bcP^0 \cup \bcP^1}$, $m_\bcR=m_{\bcP^0 \cup \bcP^1}$, and $m(\bcR) = \sum_x m_\bcR(x)$. The \defn{slice class} of $\bcR$ at $x_0 \in \bI$ is
\[ \sigma_\bcR(x_0) := - \sum_{x <x_0} m_{\bcP^0}(x) \cdot v_{\bcP^0}(x) + \sum_{x<x_0} m_{\bcP^1}(x) \cdot v_{\bcP^1}(x)  \in \Lambda \]
and the \defn{support} of $\bcR$ is $\supp(\bcR) := \supp(m_\bcR) \cup \supp(\sigma_\bcR)$. We say:
\begin{enumerate}[(a)]
\item $\bcR$ is \defn{local} if $\bcP^0=\bcP^1$ and \defn{empty} if $\bcP^0=\bcP^1=0$.
\item Two regions $\bcR_1 = (\bcP^0_1,\bcP^1_1)$ and $\bcR_2 = (\bcP^0_2,\bcP^1_2)$ are \defn{composable} at $x_0$ if $\bcP^0_1,\bcP^1_1,\bcP^0_2$ and $\bcP^1_2$ are pairwise compatible and $\max(\supp\bcR_1) \le x_0 \le \min(\supp\bcR_2)$. In this case, their \defn{concatenation} $\bcR_1\bcR_2$ is $(\bcP^0_1 \cup \bcP^0_2, \bcP^1_1 \cup \bcP^1_2)$ and $\bcR_1\bcR_2$ is said to \defn{decompose} at $x_0$. We say $\bcR$ is \defn{indecomposable} if it cannot be written as a concatenation of two non-empty regions. Any $\bcR$ can be uniquely written as a concatenation $\bcR_1 \cdots \bcR_d$ where each $\bcR_i$, called a \defn{factor}, is non-empty and indecomposable.
\item A non-local indecomposable region $\bcR=(\bcP^0,\bcP^1)$ is \defn{positive} if, for each $i$, $v_{\bcP^i}(x) \times \sigma_{\bcR}(x) \ge 0$ with equality only if $c_{\bcP^i}(x) \neq  (-1)^{i+1}$. It is \defn{minimally positive} if it is positive, each non-zero $\sigma_\bcR (x)$ is primitive and, for each $i$, $v_{\bcP^i}(x) \times \sigma_\bcR(x) \le 1$ with equality only if $c_{\bcP^i}(x) \neq (-1)^i$. A general region $\bcR$ is \defn{(minimally) positive} if each of its non-local factors is.
\end{enumerate}
A \defn{decoration} $\cR$ of $\bcR = (\bcP^0,\bcP^1)$ is a pair $(\cP^0,\cP^1)$ of decorations $\cP^i$ of $\bcP^i$. We say:
\begin{enumerate}[(a)]
\item[(d)] $\cR$ is \defn{trivial} if $\cP^0= \cP^1$.
\item[(e)] $\cR$ is \defn{minimally decorated} if, for each $i$, $m_{\cP^i}^e(x) = 0$ whenever $c_{\cP^i}(x) = (-1)^i$ and $m_{\cP^i}^h(x)=0$ whenever $c_{\cP^i}(x) = (-1)^{i+1}$. It is \defn{almost minimally decorated} if $\sum_x |m_{\cP^0}^e(x)-m_{\cP_{\min}^0}^e(x)| + |m_{\cP^1}^e(x)-m_{\cP_{\min}^1}^e(x)| = 1$ where $(\cP_{\min}^0,\cP_{\min}^1)$ is the minimal decoration of $\bcR$.
\end{enumerate}
\end{definition}
A decorated path $\cP$ and a decorated region $\cR$ inherit terminologies and operations of their underlying undecorated $\bcP$ and $\bcR$. One caveat is that a decomposition of $\cR$ at $x_0$ is unique only up to re-distributing elliptic/hyperbolic multiplicities at $x_0$. We depict $\bcR=(\bcP^0,\bcP^1)$ by a closed (not necessarily embedded) polygon between $\bcP^0$ and $\bcP^1$.

In Figure \ref{fig:diff-examples}(b), three lattice regions are drawn with $\cP^0,\cP^1$ and $\sigma_\cR$ in red, blue and dotted arrows, respectively. They are minimally positive: each triangle formed by $v_{\cP^i}(x)$ and $\sigma_{\cR}(x)$ is either degenerate or primitive with the right orientation; and the convexity $c_{\cP^i}$ satisfies the requirement, e.g. for the third region, $v_{\cP^1}(x_2)$ is not parallel to $\sigma_\cR(x_2)$ so $c_{\cP^1}(x_2)=1$ while $v_{\cP^1}(x_4)$ is parallel to $\sigma_\cR(x_4)$ so $c_{\cP^1}(x_4)=-1$, and so on. They are also almost minimally decorated: we recover the minimal decoration by reversing ellipticity of $\hat e_{x_4}$, $\check e_{x_6}$ and $\check h_{x_2}$ in each respective region.

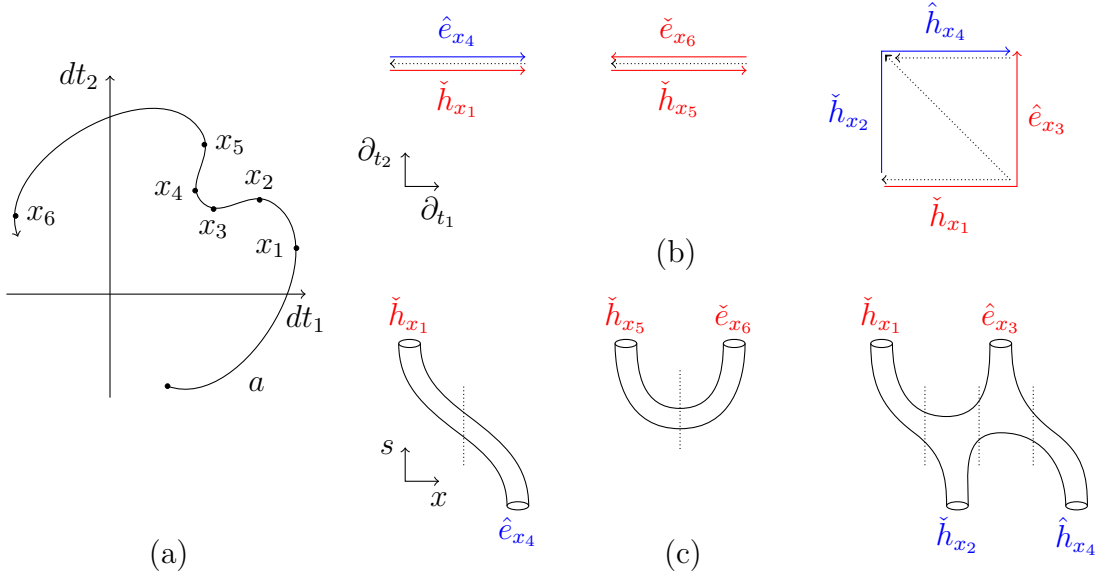
\begin{figure}
\begin{center}
\begin{tikzpicture}[scale=1.8]
\begin{scope}[xshift=-17ex,yshift=-10ex,scale=1.7]
\draw[->] (.15,0.6) -- (1.45,0.6) node[below] {$dt_1$};
\draw[->] (.6,.15) -- (.6,1.55) node[left] {$dt_2$};
\draw[->] (.85,.2) to[out=-20,in=-30] node[pos=.3, below right] {$a$} (1.3,1) to[out=150,in=-45] (1,1) to[out=135,in=-60] (1,1.3) to[out=120,in=110] (.2,0.85);

\draw[fill] (1.41,.8) node[left]{$x_1$} circle(.01);
\draw[fill] (1.25,1.01) node[above]{$x_2$} circle(.01);
\draw[fill] (1.05,.97) node[below]{$x_3$} circle(.01);
\draw[fill] (.97,1.05) node[left]{$x_4$} circle(.01);
\draw[fill] (1.01,1.25) node[right]{$x_5$} circle(.01);
\draw[fill] (.19,.94) node[right]{$x_6$} circle(.01);
\draw[fill] (.85,.2) circle(.01);
\end{scope}
\begin{scope}[yshift=5ex]
\draw[color=red,->] (0,-.05) -- (1,-.05) node[midway,below] {$\check h_{x_1}$};
\draw[densely dotted,->] (1,0) -- (0,0);
\draw[color=blue,->] (0,.05) -- (1,.05) node[midway,above] {$\hat e_{x_4}$};
\end{scope}
\begin{scope}[xshift=9ex,yshift=5ex]
\draw[color=red,->] (0,-.05) -- (1,-.05) node[midway,below] {$\check h_{x_5}$};
\draw[densely dotted,->] (1,0) -- (0,0);
\draw[color=red,->] (1,.05) -- (0,.05) node[midway,above] {$\check e_{x_6}$};
\end{scope}
\begin{scope}[xshift=20ex]
\draw[->,color=red] (.02,0) -- (1,0) node[midway,below] {$\check h_{x_1}$} -- (1,1) node[midway,right] {$\hat e_{x_3}$};
\draw[->,color=blue] (0,.1) -- (0,1) node[midway,left] {$\check h_{x_2}$}-- (.95,1) node[midway,above] {$\hat h_{x_4}$};
\draw[densely dotted,->] (.95,.05) -- (0,.05);
\draw[densely dotted,->] (.95,0.05) -- (0.03,.97);
\draw (.07,.97) -- (.03,.97) -- (.03,.93);
\draw[densely dotted,->] (.95,.95) -- (.1,.95);
\end{scope}

\begin{scope}[xshift=2ex,scale=.5]
\draw[<->] (0,0)  node[below] {$\partial_{t_1}$} -- (-.5,0) -- (-.5,.5) node[left] {$\partial_{t_2}$};
\draw (3.5,-1) node {(b)};
\end{scope}
\begin{scope}[yshift=-12ex,xshift=2ex,scale=.5]
\draw[<->] (0,0)  node[below] {$x$} -- (-.5,0) -- (-.5,.5) node[left] {$s$};
\end{scope}

\begin{scope}[yshift=-13ex,yscale=.6,xscale=.8]
\begin{scope}[xshift=1ex]
\draw (0,2) ellipse(.1 and .05) node[above,color=red] {$\check h_{x_1}$};
\draw (1,0) ellipse(.1 and .05) node[below,color=blue] {$\hat e_{x_4}$};
\draw (-0.1,2) to[out=-90,in=135] (.45,.9) to[out=-45,in=90] (.9,0);
\draw (0.1,2) to[out=-90,in=135] (.55,1.1) to[out=-45,in=90] (1.1,0);
\draw[densely dotted] (.5,.5) -- (.5,1.5);
\end{scope}
\begin{scope}[xshift=12ex]
\draw (0,2) ellipse(.1 and .05) node[above,color=red] {$\check h_{x_5}$};
\draw (1,2) ellipse(.1 and .05) node[above,color=red] {$\check e_{x_6}$};
\draw (0.1,2) to[out=-90,in=180] (.5,1.2) to[out=0,in=-90] (.9,2);
\draw (-0.1,2) to[out=-90,in=180] (.5,.95) to[out=0,in=-90] (1.1,2);
\draw[densely dotted] (.5,.7) -- (.5,1.7);
\end{scope}

\begin{scope}[xshift=25ex]
\draw (0,2) ellipse(.1 and .05) node[above,color=red] {$\check h_{x_1}$};
\draw (1.1,2) ellipse(.1 and .05) node[above,color=red] {$\hat e_{x_3}$};
\draw (0.1,2) to[out=-90,in=180] (.6,1.1) to[out=0,in=-90] (1,2);
\draw (.8,0) to[out=90,in=180] (1.1,.9) to[out=0,in=90] (1.7,0);
\draw (-0.1,2) to[out=-90,in=135] (.3,.9) to[out=-45,in=90] (.6,0);
\draw (1.2,2) to[out=-90,in=135] (1.6,.9) to[out=-45,in=90] (1.9,0);
\draw (.7,0) ellipse(.1 and .05) node[below,color=blue] {$\check h_{x_2}$};
\draw (1.8,0) ellipse(.1 and .05) node[below,color=blue] {$\hat h_{x_4}$};
\draw[densely dotted] (.4,.5) -- (.4,1.5);
\draw[densely dotted] (.9,.5) -- (.9,1.5);
\draw[densely dotted] (1.4,.5) -- (1.4,1.5);
\end{scope}
\end{scope}

\begin{scope}[xshift=-9ex,yshift=-15ex]
\draw(0,0) node {(a)};
\draw(3.8,0) node {(c)};
\end{scope}
\end{tikzpicture}
\vspace{-.5in}
\end{center}
\caption{(a) Graph of an orbital moment map $a$, (b) regions contributing to the differential, and (c) sketches of corresponding $J$-holomorphic curves.}
\label{fig:diff-examples}
\end{figure}

We now relate lattice paths and lattice regions to an orbital moment map $a$. As mentioned above, the Reeb vector field $\bar R$ of $\pi_\bI^*a$ as a function from $\bI$ to $\ft^2$ is
\[ \bar R = (a')^\vee/(a \times a') \]
where we identify $(\ft^2)^* \isom \ft^2$ via $(dt_1)^\vee = -\partial_{t_2}$ and $(dt_2)^\vee = \partial_{t_1}$. It is also convenient to write $u_1 \sim u_2$ when $u_1$ is a positive multiple of $u_2$ for $u_1,u_2 \in \ft^2,(\ft^2)^*$ or $\bR$ and $\delta_x:\bI \to \bR$ for the function supported at $\{x\}$ with $\delta_x(x)=1$.

\begin{definition}\label{def:assoc}
Let $a$ be a generic orbital moment map and $\lambda$ a perturbation of $\pi_\bI^*a$ described earlier (and more precisely in \S\ref{section:mb}).
\begin{enumerate}[(a)]
\item A path $\bcP$ is \defn{$a$-compatible} if $v_\bcP(x) \sim a'(x)^\vee$ and $c_\bcP(x) \sim (a' \times a'')(x)$ for each $x \in \supp\bcP$.
\item A region $\bcR$ is \defn{$a$-positive} if each $\bcP^i$ is $a$-compatible and $(a')^\vee \times \sigma_\bcR \ge 0$.
\end{enumerate}
To an orbit set $\gamma = \{(e_{x_i},m_i^e),(h_{x_j},m_j^h)\}$ of $\lambda$, we \defn{associate} $\cP_\gamma$, the unique $a$-compatible decorated path with $m_{\cP_\gamma}^e= \sum m^e_i \delta_{x_i}$ and $m_{\cP_\gamma}^h=\sum m^h_j \delta_{x_j}$. To a pair of orbit sets $\alpha$ and $\beta$ with $[\alpha] =[\beta] \in H_1(T^2)$, we \defn{associate} the decorated region $\cR_\ab=(\cP_\alpha,\cP_\beta)$.
\end{definition}
Figure \ref{fig:diff-examples} shows an orbital moment map $a$ and $a$-positive lattice regions associated to admissible orbit sets $\alpha$ and $\beta$ of $\lambda$. It also sketches $J$-holomorphic curves $C$ from $\alpha$ to $\beta$ with each dotted line showing the ``slice'' $C \cap (\bR \times  \{x_0\} \times T^2)$, whose homology class agrees with $\sigma_\cR(x_0)$. According to Theorem \ref{thm:main}, these regions correspond to non-zero differential coefficients. For non-examples, see Figure \ref{fig:diff-nonexamples}.

\begin{remark}
Write $(\cP^0,\cP^1)$ for $\cR'$ from Theorem \ref{thm:main}. If $a$ is convex everywhere, e.g. $(T^3,\lambda_n)$, it is easy to deduce from Definition \ref{def:region} that $m(\cP^0) =2$ with $\supp(\cP^0) = \partial (\supp\cR')$. This is the ``rounding corner'' operation in \cite{t3}. Similarly, if $a$ is concave everywhere, $m(\cP^1)=2$ with $\supp(\cP^1) = \partial(\supp\cR')$, a ``dual'' to rounding a corner as in \cite{pfh}. In general, if $a_1$ and $a_2$ are orbital moment maps and $a_2$ is a reflection of $a_1$ through a line of rational slope in $(\ft^2)^*$, ECC of $\pi_\bI^*a_2$ is dual to ECC of $\pi_\bI^*a_1$.
\end{remark}

{\bf Acknowledgement.} I would like to thank Michael Hutchings, who inspired and guided me through this project and my career, Andr\'as Stipsicz for his support and valuable advice and Vinicius Gripp for stimulating discussions. I thank Denis Auroux, Dan Cristofaro-Gardiner, Roman Golovko, and Klaus Niederkr\"uger for their interest and many helpful suggestions.

\section{Preliminaries}\label{section:prelim}

\subsection{Review of embedded contact homology}\label{section:ech}

We briefly review ECH following \cite{bn} (see also \cite{ir}). Let $Y$ be a 3-manifold with a non-degenerate contact form $\lambda$ and pick a generic $\lambda$-admissible almost complex structure $J$ on $\bR \times Y$. Admissibility means that $J$ is $\bR$-invariant, sends $\xi=\ker\lambda$ to itself rotating positively with respect to $d\lambda$, and $J(\partial_s)=R$ where $s$ is the $\bR$-coordinate and $R$ is the Reeb vector field.

\vspace{.1in}\noindent {\bf Generators.} An \defn{orbit set} $\gamma$ is a finite set of pairs $\{(\gamma_i,m_i)\}$ where $\gamma_i$ are distinct embedded Reeb orbits and $m_i$ are positive integers. We say $\gamma$ is \defn{admissible} if $m_i=1$ whenever $\gamma_i$ is hyperbolic and its homology class is $\sum_i m_i[\alpha_i] \in H_1(Y)$. The ECH chain complex $ECC(Y,\lambda,J)$ (or $ECC(Y,\lambda,J,\Gamma)$) is generated (over $\bZ/2$ coefficients) by admissible orbit sets (in the homology class $\Gamma$).

\vspace{.1in}\noindent {\bf Holomorphic currents.} Consider $J$-holomorphic curves in $(\bR \times Y,J)$ with positive and negative ends at Reeb orbits. Two $J$-holomorphic curves $C$ and $C'$ are said to be \defn{equivalent} if $C$ is obtained from $C'$ by a pre-composition with a biholomorphic map on its domain. Then, a \defn{$J$-holomorphic current} $\cC$ is a finite set of pairs $\{(C_k,d_k)\}$ where $C_k$ are equivalent classes of distinct irreducible somewhere injective $J$-holomorphic curves and $d_k$ are positive integers. The \defn{moduli space} $\cM^J(\ab)$ (or $\cM(\ab)$) of $J$-holomorphic currents from $\alpha = \{(\alpha_i,m_i)\}$ to $\beta = \{(\beta_j,n_j)\}$ consists of $J$-holomorphic currents whose total multiplicity of positive ends at covers of $\alpha_i$ is $m_i$ and whose total multiplicity of negative ends at covers of $\beta_i$ is $n_i$, with no other ends. The homology class of $\cC$ is $\sum_k d_k[C_k]$. We say that $\cC$ is \defn{somewhere injective} if $d_k = 1$ for each $k$ and that $\cC$ is \defn{embedded} if it is somewhere injective, each $C_k$ is embedded and $C_k$ are pairwise disjoint.

\vspace{.1in}\noindent {\bf The ECH index.} For $\alpha$ and $\beta$ as above, let $H_2(Y,\alpha,\beta)$ denote the set of 2-chains $Z$ in $Y$ with $\partial Z = \sum_i m_i\alpha_i - \sum_j n_j\beta_j$, modulo boundaries of 3-chains. Fix a symplectic trivialization $\tau$ of $\xi$ over each $\alpha_i$ and $\beta_j$. The \defn{ECH index} for the triple $(\alpha,\beta,Z)$ is
\begin{equation}\label{eq:echind}
I(\alpha,\beta,Z) := c_\tau(Z) + Q_\tau(Z) + CZ_\tau^I(\alpha,\beta) \in \bZ.
\end{equation}
Here,
\begin{equation*} \label{eq:czech}
 CZ_\tau^I(\alpha,\beta) := \sum_i \sum_{k=1}^{m_i} CZ_\tau(\alpha_i^k) - \sum_j \sum_{l=1}^{n_j} CZ_\tau(\beta_j^l)
\end{equation*}
where $CZ_\tau(\rho) \in \bZ$ denotes the Conley-Zehnder index of an orbit $\rho$ with respect to $\tau$. If $S$ is an embedded surface representative of $Z$, the \defn{relative Chern class} $c_\tau(Z)=\<c_1(\xi,\tau),Z\>$ is the count of zeroes of a section of $\xi|_S$ which is constant with respect to $\tau$ near each of its ends. The \defn{relative intersection pairing} $Q_\tau(Z)$ is the count of intersections of two embedded (except at the boundary) transversely intersecting surfaces $S$ and $S'$ in $[-1,1] \times Y$ subject to the following: (i) $S$ and $S'$ represent $Z$ and $\partial S = \partial S' = \sum_i m_i(\{1\} \times \alpha_i) - \sum_j n_j (\{-1\} \times \beta_j)$ and (ii) the projection of $(S \cup S') \cap ((1-\epsilon,1) \times Y)$ to $Y$ is an embedding, and its image in a transverse slice to $\alpha_i$ is a union of rays which do not intersect and which do not rotate with respect to $\tau$ as one goes around $\alpha_i$ (and similarly for $(S \cup S') \cap ((-1,-1+\epsilon) \times Y)$).

We remark that $I(\alpha,\beta,Z)$ does not depend on the choice of $\tau$. If $C$ (or $\cC$) is a $J$-holomorphic curve (current) from $\alpha$ to $\beta$ in the homology class $Z \in H_2(Y,\alpha,\beta)$, we also write $I(C)$ (or $I(\cC)$) for $I(\alpha,\beta,Z)$. Compare \eqref{eq:echind} with the Fredholm index
\begin{equation}\label{eq:fredind}
 \ind(C) = -\chi(\Sigma) + 2c_\tau([C]) + \sum CZ_\tau(\rho_i^+) - \sum CZ_\tau(\rho_j^-),
\end{equation}
where $\Sigma$ is the domain of $C$ and the two sums are over its positive ends at $\rho_i^+$ and negative ends at $\rho_j^-$, respectively.

\begin{prop}(\cite[\S 3]{bn}) Let $\alpha,\beta$ and $\gamma$ be orbit sets of $\lambda$ in the homology class $\Gamma$.
\begin{enumerate}[(a)]
\item If $Z,Z' \in H_2(Y, \alpha,\beta)$ and $W \in H_2(Y,\beta,\gamma)$,
\begin{equation}\label{eq:indamb}
I(\alpha,\beta,Z) - I(\alpha,\beta,Z') = \<c_1(\xi) + 2PD(\Gamma), Z-Z'\>
\end{equation}
where $PD$ denotes the Poincare dual and
\[ I(\alpha,\gamma,Z+W) = I(\alpha,\beta,Z) + I(\beta,\gamma,W). \]
\item If $C \in \cM(\ab)$ is somewhere injective,
\begin{equation} \label{eq:indineq}
 \ind(C) \le I(C)
\end{equation}
with equality only if $C$ is embedded and satisfies the partition condition below.
\item If $\cC \in \cM(\ab)$ contains no trivial cylinders and $\cT$ is a union of trivial cylinders,
\begin{equation}\label{eq:trivcyl}
I(\cC \cup \cT) \ge I(\cC) + 2\# (\cC \cap \cT).
\end{equation}
\end{enumerate}
\end{prop}

\vspace{.1in}\noindent {\bf Partition conditions.} Let $C$ be a $J$-holomorphic curve from $\alpha= \{(\alpha_i,m_i)\}$ to $\beta= \{(\beta_j,n_j)\}$. For each $i$, $C$ has ends at covers of $\alpha_i$ with total multiplicity $m_i$. The multiplicities of these covers form a partition of $m_i$, which we denote by $p_i^+(C)$. We similarly define the partition $p_j^-(C)$ of $n_j$ for each $j$.

For each embedded Reeb orbit $\rho$ and each positive integer $m$, we define two special partitions $p_\rho^+(m)$ and $p_\rho^-(m)$ of $m$. If $\rho$ is positive hyperbolic, then $p_\rho^+(m)=p_\rho^-(m) = (1, \cdots, 1)$. If $\rho$ is elliptic with rotation angle $\phi$, let $\Lambda_\phi^+(m)$ be the maximal concave polygonal path in the $x,y$-plane with vertices at lattice points which starts at the origin, ends at $(m,\floor{m\phi})$ and lies below the line $y=\phi x$. Then, $p_\rho^+(m) = p_\phi^+(m)$ consists of the horizontal displacements of the segments of $\Lambda_\phi^+(m)$ connecting consecutive lattice points; and $p_\rho^-(m):=p_{-\phi}^+(m)$. (For more details or $p_\rho^\pm$ for negative hyperbolic $\rho$, see \cite[\S 3.9]{bn}.) Any $C$ satisfying equality in \eqref{eq:indineq} must satisfy $p_i^+(C) = p^+_{\alpha_i}(m_i)$ and $p_j^-(C) = p_{\beta_j}^-(n_j)$ for each $i$ and $j$.

\vspace{.1in}\noindent {\bf Differentials.} Let $\cM_k(\ab) := \{\cC \in \cM(\ab)|I(\cC)=k\}$. The key consequence of \eqref{eq:indineq} and \eqref{eq:trivcyl} is that, if $J$ is generic, any $\cC \in \cM_1(\alpha,\beta)$ can be written as the disjoint union $C' \sqcup \cT$ where $\cT$ is trivial and $C'$ is an irreducible embedded $J$-holomorphic curve with $I(C')=\ind(C') = 1$. We also have that $\cM_1(\alpha,\beta)/\bR$ is compact by a version of Gromov compactness (See \cite[\S 5.3]{bn}). If $\alpha$ and $\beta$ are admissible, we define
\[ \< \partial \alpha, \beta\> := \#(\cM_1(\alpha,\beta)/\bR) \in \bZ/2. \]

\vspace{.1in}\noindent {\bf Filtration.} The action $\cA(\alpha)$ of an orbit set $\alpha = \{(\alpha_i,m_i)\}$ is
\[ \cA(\alpha) := \sum_i m_i \int_{\alpha_i} \lambda. \]
By Stokes' theorem, the ECH chain complex is filtered by the action of its generators. For each $L > 0$, the filtered ECH chain complex $ECC^L$ is generated only by orbit sets with action less than $L$. In this paper, we formulate the filtered ECC so that there is a natural chain inclusion map $ECC^L(Y,\lambda_L,J_L) \to ECC^{L'}(Y,\lambda_{L'},J_{L'})$ for $L< L'$. We recover the full ECH as the direct limit of $ECH^L$ as $L \to \infty$. With this understood, we drop $L$ from the notation.

\subsection{Morse-Bott theory}\label{section:mb}

We return to $Y= \bI \times T^2$ with a contact form $\pi_\bI^*a$. To define ECC, we perturb $\pi_\bI^*a$ to a non-degenerate $\lambda$ and choose a generic $\lambda$-admissible almost complex structure $J$ on $\bR \times Y$. The goal of this section is to describe $\lambda$ and $J$ for which ECC yields a nice combinatorial description.

Define $\bar Q:\bI \to \ft^2, v_a:\bI \to \Lambda$ and $\cA_a:\bI \to \bR^+ \cup \{\infty\}$ by $\bar Q := -a^\vee$; $v_a(x)=v$ if $a'(x)^\vee \sim v$ for a primitive $v \in \Lambda$ and 0 otherwise; and $\cA_a(x)= \frac{(a \times a')(x)}{\norm{a'(x)}} \norm{v_a(x)}$ if $v_a(x) \neq 0$ and infinity otherwise. Let $\Xi_L:= \{x \in \bI|\cA_a(x) < L\}$ and $N:= L/\min\cA_a$. Then, for $\rho$ in some $S^1$-family $\brho_x$, $\cA(\rho)=\cA_a(x)$, $\cL_{\partial_x} \bar R =  \frac{a'\times a''}{(a \times a')^2}\bar Q$ and $\cL_{\bar Q}\bar R =0$, so the linearized Reeb flow over $\rho$ is contained in the Maslov cycle with the return map
\begin{equation}\label{eq:returnmap}
 \pmat{1 & 0 \\ \frac{a'\times a''}{(a \times a')^2}\cA_a & 1},
\end{equation}
while $\bar Q$ describes a section of $\xi$ with $(d\lambda)(\partial_x,\bar Q) =a \times a'> 0$, giving a trivialization $\tau$ of $\xi$ by
\[ \xi \isom \Span\{\partial_x,\bar Q\}. \]

\vspace{.1in}\noindent{\bf Perturbation.} (cf. \cite[\S 10.5]{t3}) Whenever $v_a(x)=(p,q) \neq 0$, define $\Theta_{x}:\{x\} \times T^2 \to \bR/\bZ$ by
\[  \Theta_{x}(x,t_1,t_2) = (t_1,t_2) \times (p,q) + pq/2.  \]
For each $\theta \in \bR/\bZ$, $\Theta_{x}^{-1}(\theta)$ is the image of a unique embedded orbit in $\brho_{x}$, which we denote by $\brho_{x}(\theta)$. Fix $\eta > 0$, $0<\theta_h<1/5$ and $\theta_e:= -\theta_h/N$. For each $x \in \Xi_L$, let $f_{x}:\bR/\bZ \to \bR$ be a Morse function which attains maximum at $\theta_{x}^{\max}$ and minimum at $\theta_{x}^{\min}$ with no other critical points, where $(\theta_{x}^{\max},\theta_{x}^{\min}) = (\theta_e,\theta_h)$ if $\brho_{x}$ is convex and $(-\theta_h,-\theta_e)$ otherwise. Then, choose disjoint neighbourhoods $U_{x}$ of $x$ so that $a' \times a''$ does not vanish on $U_{x}$ and extend $\Theta_{x}^*(f_{x})$ to a function $\tilde f_{x}$ on $U_{x} \times T^2$ with a compact support and $\partial \tilde f_{x}/\partial x = 0$ near $\{x\} \times T^2$. If $\eta$ is sufficiently small,
\begin{equation}\label{eq:mblambda}
  \lambda := (1 + \eta\tilde f_{x} )\pi_\bI^*a
\end{equation}
is a contact form on $U_{x} \times T^2$ with non-degenerate embedded orbits $\brho_{x}(\theta_e)$ and $\brho_{x}(\theta_h)$ and no other embedded orbits of action less than $L$ \cite{mb}. By \eqref{eq:returnmap}, if $\brho_{x}$ is convex, $\check e_{x}:= \brho_{x}(\theta_e)$ is elliptic, $\check h_{x}:=\brho_{x}(\theta_h)$ is hyperbolic and their $m$-fold iterates for $m<N$ have $CZ_\tau(\check e_{x}^m) = 1$ and $CZ_\tau(\check h_{x}^m) = 0$, provided $\eta$ is sufficiently small. Similarly, if $\brho_{x}$ is concave, $\hat e_{x}:=\brho_{x}(\theta_e)$ is elliptic with $CZ_\tau(\hat e_{x}^m) = -1$ and $\hat h_{x}:=\brho_{x}(\theta_h)$ is hyperbolic with $CZ_\tau (\hat h_{x}^m) = 0$.

\begin{definition}\label{def:goodperturbation}
We say a perturbation $\lambda$ of $\pi_\bI^*a$ is \defn{good} if it is of the form \eqref{eq:mblambda} on $U_x \times T^2$ for each $x \in \Xi_L$ and unperturbed elsewhere.
\end{definition}

\vspace{.1in}\noindent{\bf Holomorphic building.} Define an almost complex structure $\bar J$ on $\bR \times \bI \times T^2$ by $\bar J(\partial_s) = \bar R$ and $\bar J(\partial_x) = \bar Q$. For $(Y, \pi_\bI^*a, \bar J,\{f_x\}_{x \in \Xi_L})$, a \defn{$\bar J$-holomorphic building} $\bar C$ is a sequence of $\bar J$-holomorphic curves $(\bar C^1, \cdots, \bar C^l)$ such that:
\begin{enumerate}[(i)]
\item Each end of $\bar C^i$ converges to the $m$-fold iterate $\brho^m_x(\theta)$ of some $\brho_x(\theta)$.
\item For $1 \le i < l$, there is a bijective pairing between the negative ends of $\bar C^i$ and the positive ends of $\bar C^{i+1}$. For each such pair, the negative end of $\bar C^i$ converges to $\brho^m_x(\theta^-)$, the positive end of $\bar C^{i+1}$ converges to $\brho^m_x(\theta^+)$ for the same $\brho_x$ and $m$ and there is a downward flow of $f_x$ from $\theta^-$ to $\theta^+$.
\item For each positive end of $\bar C^1$ at some $\brho^m_x(\theta^+)$, there is a downward flow of $f_x$ from a critical point of $f_x$ to $\theta^+$. For each negative end of $\bar C^l$ at some $\brho^m_x(\theta^-)$, there is a downward flow of $f_x$ from $\theta^-$ to a critical point of $f_x$.
\end{enumerate}

\section{Proof of the main theorem}\label{section:main}

\begin{definition}\label{def:combind}
The \defn{local combinatorial ECH index} of a decorated region $\cR$ at $x$ is $ I_\cR(x) := Q_\cR(x) + CZ_\cR(x)$ where
\[ Q_\cR(x) = m_{\cR}(x)\cdot (v_\cR(x) \times \sigma_\cR(x)) \]
and
\[ CZ_\cR(x) = c_{\cP^0}(x) \cdot m_{\cP^0}^e(x) - c_{\cP^1}(x) \cdot m_{\cP^1}^e(x). \]
The \defn{combinatorial ECH index} $I(\cR)$ of $\cR$ is $\sum_x I_\cR(x)$.
\end{definition}
Note $\sum_x Q_\cR(x)$ is the area of the polygon depicting $\cR$ with respect to the standard area form. As one might expect:

\begin{prop}\label{prop:index}
Let $a$ be a generic orbital moment map and $\lambda$ a good perturbation of $\pi_\bI^*a$. For orbit sets $\alpha$ and $\beta$ of $\lambda$ with $[\alpha]=[\beta]$ and any $Z \in H_2(Y,\alpha,\beta)$,
\[ I(\cR_\ab) = I(\alpha,\beta,Z). \]
\end{prop}

\begin{proof}
Since $\xi$ is trivial and the generator $[T^2] \in H_2(Y)$ has algebraic intersection number zero with every orbit, $c_\tau(Z) =0$ and $I(\alpha,\beta,Z)$ is independent of $Z$ by \eqref{eq:indamb}. It is also clear that $\sum_x CZ_{\cR_\ab}(x) = CZ_\tau^I(\alpha,\beta)$ (see \S \ref{section:mb}). To compute $Q_\tau(Z)$, let
\begin{equation}\label{eq:graph}
 G := (\{0\} \times \bI) \cup \{(s, x + (1-\abs{s})\epsilon)\}_{s \in [-1,1],x \in \supp(\cP_\alpha \cup \cP_\beta)}
\end{equation}
be a union of line segments in $[-1,1] \times \bI$ with multivalent vertices
\[ V: = \{(0,x+\epsilon)\}_{x \in \supp(\cP_\alpha \cup \cP_\beta)}. \]
Let $B_{\epsilon/2}(V)$ be the $(\epsilon/2)$-neighbourhood of $V$ and $\pi := \id_{[-1,1]} \times \pi_\bI$. We want a smooth surface $S \subset [-1,1] \times Y$ as in \S \ref{section:ech} such that (i) $\pi(S) \subset G \cup B_{\epsilon/2}(V)$ and (ii) for each component $E$ of $G\setminus B_{\epsilon/2}(V)$, $\pi|_S^{-1}(E)$ consists of minimal number of embedded (disjoint except at $\{\pm 1\} \times Y$) $v$-invariant cylinders, where $v=\sigma_{\cR_\ab}(x)$ if $(0,x) \in E$ and $v_a(x)$ if $(\pm 1, x) \in E$. We can construct such an $S$ by gluing these cylinders so that, near each $(0,x_0) \in V$, the projection $\cZ_x$ of $S \cap ([-1,1]\times \{x\} \times T^2)$ to $T^2$ is a movie of curves with $\cZ_{x_0+ \epsilon/2}$ obtained from $\cZ_{x_0-\epsilon/2}$ by resolving intersections.

If $\psi$ is an automorphism of $[-1,1]$ with $\psi(-1+\epsilon)=0$ and $\psi(0)=1-\epsilon$, $S$ and $S' := (\psi\times \id_Y)(S)$ intersect in $\pi^{-1}(1-\epsilon,x+\epsilon^2)$ with signed count $m_{\cP_\alpha}(x)\cdot (v_{\cP_\alpha}(x) \times \sigma_{\cR_\ab}(x))$ for each $x \in \supp\cP_\alpha$, and in $\pi^{-1}(0,x+\epsilon^2)$ with signed count $m_{\cP_\beta}(x)\cdot (v_{\cP_\beta}(x) \times \sigma_{\cR_\ab}(x))$ for each $x \in \supp\cP_\beta$. We get $\sum_x Q_{\cR_\ab}(x) = Q_\tau(Z)$ by summing up these numbers.
\end{proof}

\subsection{Positivity}\label{section:positivity}

\begin{lem}\label{lem:pos}
Let $a$ be a generic orbital moment map and $\cR$ a lattice region. 
\begin{enumerate}[(a)]
\item If $\cR$ is positive, indecomposable and non-local, $I_\cR \ge 0$.
\item If $\bcR$ is $a$-positive, then it is positive and it decomposes at $x$ whenever $a'(x)^\vee \times \sigma_\bcR(x) = 0$.
\end{enumerate}
\end{lem}

\begin{proof}
Part (a) is clear from definition. For (b), we write $\bcR=(\bcP^0,\bcP^1)$ and show that if $\bcR$ is non-local and indecomposable, then: (i) $Q_\bcR \ge 0$, (ii) $(a')^\vee \times \sigma_\bcR > 0$ on $int(\supp\bcR)$ and (iii) $c_{\bcP^0}(x) \ge 0$ and $c_{\bcP^1}(x) \le 0$ for $x \in \partial(\supp\bcR)$. Claim (i) follows from $a$-positivity at $x \in \supp(m_\bcR)$. For (ii) and (iii), suppose $a'(x_0)^\vee \times \sigma_\bcR(x_0) = 0$ for $x_0 \in \supp(\bcR)$ and write $\sigma_\bcR(x_0\pm \epsilon) = b_\pm \cdot a'(x_0)^\vee$. By $a$-positivity, $a'(x_0\pm \epsilon) \times b_\pm a'(x_0) \ge 0$, while $a'(x_0) \times a'(x_0 \pm \epsilon) \sim \pm (a' \times a'')(x_0)$. In particular, $b_+ b_- \le 0$ so $x_0 \in \partial(\supp\bcR)$ by indecomposability. Furthermore, if $a$ is convex at $x_0$, $b_+ - b_- <0$ so $m_{\bcP^0}(x_0) > 0$ and by non-locality, $m_{\bcP^1}(x_0) = c_{\bcP^1}(x_0) = 0$. Similarly, if $a$ is concave at $x_0$, $c_{\bcP^0}(x_0) = 0$.
\end{proof}

\begin{definition}\label{def:barassoc}
Let $a$ be a generic orbital moment map and let $\lambda$ be $\pi_\bI^*a$ or a good perturbation thereof. To each orbit set $\gamma$ of $\lambda$, we \defn{associate} $\bcP_{\gamma}$, the unique $a$-compatible path such that $\gamma$ contains $m_{\bcP_\gamma}(x)$ orbits (counted with multiplicity) at $x$. To orbit sets $\alpha$ and $\beta$ with $[\alpha]=[\beta]$, we \defn{associate} $\bcR_\ab := (\bcP_\alpha,\bcP_\beta)$.
\end{definition}

\begin{lem}\label{lem:apos}
Let $\lambda$ be as in Definition \ref{def:barassoc} and $J$ a generic $\lambda$-admissible almost complex structure. If $\alpha$ and $\beta$ are orbit sets of $\lambda$ and $\cM(\ab)$ is nonempty, $\bcR_{\alpha,\beta}$ is $a$-positive. Moreover, if $C \in \cM(\ab)$ has no end at $x \in \bI$ and $\lambda$ is unperturbed near $\{x\} \times T^2$, then $a'(x)^\vee \times \sigma_{\bcR_\ab}(x)=0$ if and only if $\cS_x:=C \cap (\bR \times \{x\} \times T^2) = \emptyset$.
\end{lem}

\begin{proof}
If $\lambda$ is unperturbed near $\{x_0\} \times T^2$ and $C$ has no end at $x_0$, differentiate
\begin{equation}\label{eq:pos}
\<a(x) - a(x_0),[\cS_{x_0}]\> = \int_{C \cap (\bR \times [x,x_0] \times T^2)} d\lambda
\end{equation}
near $x_0$ to get $a'(x)^\vee \times \sigma_\bcR(x) = \<a'(x),\sigma_\bcR(x)\> \ge 0$ with equality only if $\cS_{x_0} = \emptyset$ by genericity of $a$. The inequality extends to all $x$ by continuity and since $a' \times a''$ does not vanish on $U_x$.
\end{proof}

\subsection{Indecomposability}\label{section:indecomp}

Let $\lambda$ be a good perturbation of $\pi_\bI^*a$ for a generic orbital moment map $a$ and $J$ a generic $\lambda$-admissible almost complex structure.

\begin{prop}\label{prop:indecomposable}
Let $\alpha$ and $\beta$ be orbit sets of $\lambda$ and $C \in \cM_1(\ab)$. Then its non-trivial component $C' \in \cM_1(\alpha',\beta')$ has genus 0 and $\cR_\ab = \cT_1 \cR_{\alpha',\beta'} \cT_2$ where $\cR_{\alpha',\beta'}$ is indecomposable and $\cT_i$ are trivial. Moreover, we have a bijection
\[ \cM_1(\alpha,\beta) \isom \cM_1(\alpha',\beta'). \]
\end{prop}

\begin{proof}
By the partition condition, we can rewrite \eqref{eq:fredind} for $C'$ as
\begin{align}\label{eq:indecomp}
1= \ind(C') &=2g - 2 +  \sum (1+ CZ_\tau(\rho_i^+)) + \sum (1-CZ_\tau(\rho_j^-)) \nonumber \\
&= 2g-2 + \sum (m_{\cR'}(x) + CZ_{\cR'}(x)).
\end{align}
where $g$ is the genus of $C'$ and $\cR'=\cR_{\alpha',\beta'}$. Write $\supp\cR'=[x_1,x_2]$. If $x_1=x_2$, $CZ_{\cR'}(x_1)=I(\cR')=1$ so $2g + m(\cR')\le 2$, forcing $g=0$ and $m(\cR')=2$. Otherwise, for each factor $\cR_i'$ of $\cR'$ and $x \in \partial(\supp\cR_i')$, $m_{\cR'_i}(x)+CZ_{\cR'_i}(x) \ge 1$, so $g=0$ and $\cR'$ contains one non-local factor and possibly one local factor. We draw a contradiction when it contains a local factor $\cR'_{i_0}$. By symmetry, we only argue for the case $\supp\cR'_{i_0}=\{x_1\}$ and $c_{\cR'}(x_1)=1$, and by $SL_2(\bZ)$-symmetry, assume $v_{\cR'}(x_1)=(0,1)$. Let $S:= C' \cap ([-s,s] \times [0,x_1+\epsilon] \times T^2)$ for generic $s, 1/\epsilon \gg 0$. Since $I(\cR') = I(\cR_{\check e_{x_1},\check e_{x_1}}\cR')$, $S$ does not intersect $\bR \times \check e_{x_1}$ by \eqref{eq:trivcyl} and it maps to 
\[ Y' := ([0,x_1+\epsilon] \times S^1)/ (\{0\} \times S^1)  \setminus \{(x_1,\theta_e)\} \]
by $(s,x,t_1,t_2) \mapsto (x,t_1)$. Let $\cS_1, \cdots, \cS_n$ denote the boundary components of $S$ with $\cS_1\subset \{-s\} \times Y$ corresponding to the unique negative end of $C'$ at $\check e_{x_1}$. Since (the two) $\cS_i \subset \{s\} \times Y$ maps to a neighbourhood of $(x_1,\theta_h)$ and $\sigma_{\cR'}(x_1+\epsilon) = (0,-1)$, the total degree of $\cup_{i=2}^n \cS_i$ mapping to $Y' \simeq S^1$ is zero. Hence, $\cS_1$ has winding number zero around $(x_1,\theta_e)$, contradicting its lower bound of $\ceil{CZ_\tau(\check e_{x_1})/2}$ from the decay condition \cite{hwz}.

If $C$ contains a trivial cylinder at $x_0$, $I(\cR_\ab)=I(\cR')$ implies $Q_{\cR_\ab}(x_0) = Q_{\cR'}(x_0)$. By Lemma \ref{lem:pos}(b), $x_0 \not \in int(\supp\cR')$ and $\cR_\ab = \cT_1\cR'\cT_2$ for some $\cT_i$. To see the bijection, each $\cC \in \cM_1(\alpha',\beta')$ gives a distinct member of $\cM_1(\ab)$ by unioning with a trivial current. This mapping is onto since the nontrivial component of any $C \in \cM_1(\ab)$ cannot have both a positive and a negative end at $x \in \supp(\cT_i)$ by the above.
\end{proof}

\subsection{Classification of $\cR_\ab$ with nonempty $\cM_1(\ab)$}\label{section:classify}

\begin{definition}\label{def:mbcomb}
\begin{enumerate}[(a)]
\item The \defn{(local) Morse-Bott ECH index} of a region $\bcR$ is $I(\cR_{\min})$ (or $I_{\cR_{\min}}$) where $\cR_{\min}$ is the minimal decoration of $\bcR$. The \defn{(local) Morse-Bott ECH index} of a decorated region is that of its underlying undecorated region.
\item The \defn{loose multiplicity} $m^l_\bcR(x)$ of $\bcR = (\bcP^0,\bcP^1)$ at $x$ is $m_{\bcP^0}(x)$ if $c_\bcR(x)=1$ and $m_{\bcP^1}(x)$ otherwise. We also write $m^l(\bcR) = \sum_x m^l_{\bcR}(x)$.
\end{enumerate}
\end{definition}

\begin{figure}
\begin{center}
\begin{tikzpicture}
\begin{scope}[scale=1.5]
\draw[->,color=red] (0,0) -- (1,0) node[midway,below] {$\check h$} -- (1,1) node[midway,right] {$\check h$} -- (0,1) node[midway,above] {$\check h$};
\draw[->,color=blue] (0,0) -- (0,.9) node[midway,left] {$\check e$};
\draw[densely dotted,->] (1,.96) -- (0,.96);
\draw[densely dotted,->] (1,.04) -- (0,.04);
\end{scope}
\begin{scope}[xshift=44ex,scale=1.5]
\draw[color=red,->] (0,0) -- (1,0) node[midway,below] {$\check e$};
\draw[color=blue,->] (0,.04) -- (1,.04) node[midway,above] {$\check h$};
\end{scope}
\begin{scope}[xshift=17ex,scale=1]
\draw (0,2) ellipse(.1 and .05) node[above,color=red] {$\check h$};
\draw (.8,2) ellipse(.1 and .05) node[above,color=red] {$\check h$};
\draw (1.6,2) ellipse(.1 and .05) node[above,color=red] {$\check h$};
\draw (.8,0) ellipse(.1 and .05) node[below,color=blue] {$\check e$};
\draw (0.1,2) to[out=-90,in=180] (.45,1.3) to[out=0,in=-90] (.7,2);
\draw (0.9,2) to[out=-90,in=180] (1.15,1.3) to[out=0,in=-90] (1.5,2);
\draw (-0.1,2) to[out=-90,in=135] (.3,1) to[out=-45,in=90] (.7,0);
\draw (1.7,2) to[out=-90,in=45] (1.3,1) to[out=-135,in=90] (.9,0);
\draw[densely dotted] (.4,.5) -- (.4,1.5);
\draw[densely dotted] (1.2,.5) -- (1.2,1.5);
\draw[very thin] (-.3,.4) -- (1.9,1.6);
\draw[very thin] (-.3,1.6) -- (1.9,.4);
\end{scope}
\begin{scope}[xshift=60ex,scale=1]
\draw (0,2) ellipse(.1 and .05) node[above,color=red] {$\check e$};
\draw (0,0) ellipse(.1 and .05) node[below,color=blue] {$\check h$};
\draw (-0.1,2) to[out=-90,in=100] (0,1) to[out=-80,in=90] (-.1,0);
\draw (0.1,2) to[out=-90,in=100] (.2,1) to[out=-80,in=90] (.1,0);
\draw (-0.1,2) to[out=-90,in=90] (-.25,1) to[out=-90,in=90] (-.1,0);
\draw (0.1,2) to[out=-90,in=90] (-.05,1) to[out=-90,in=90] (.1,0);
\end{scope}
\end{tikzpicture}
\end{center}
\vspace{-.3in}
\caption{Examples of case (c2) and (c3)}
\label{fig:diff-nonexamples}
\end{figure}
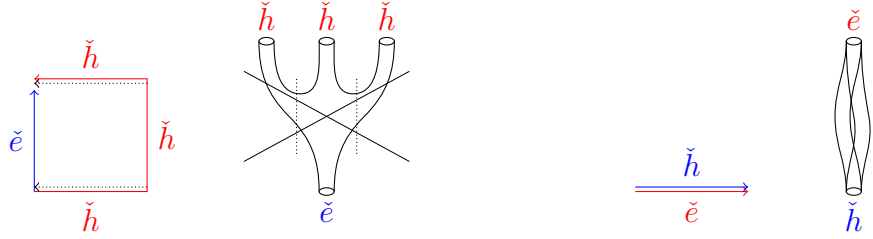

Let $\alpha$ and $\beta$ be orbit sets of $\lambda$ as in \S\ref{section:indecomp}. To study $\cM_1(\alpha,\beta) \neq \emptyset$, it suffices to assume $\cR=\cR_\ab$ is $a$-positive and indecomposable by Lemma \ref{lem:apos} and Proposition \ref{prop:indecomposable}. Let $\cR_{\min}$ be the minimal decoration of the underlying undecorated region of $\cR$ so $I(\cR) \ge \bar I(\cR)=I(\cR_{\min}) \ge -1$ by Lemma \ref{lem:pos}(a). Hence, we classify $\cR$ by:
\begin{enumerate}[(c1)]
\item $\bar I(\cR) = 0$: By positivity of $I_{\cR_{\min}}$, $v_\cR \times \sigma_\cR \le 1$. If some $\sigma_\cR(x)$ is non-primitive, $Q_\cR \equiv 0$ so $\sum (CZ_\cR(x)+m_\cR(x)) \ge 4$, violating \eqref{eq:indecomp}. Thus, $\cR$ is minimally positive and almost minimally decorated. We claim $\#(\cM_1(\alpha,\beta)/\bR) = 1$.
\item $\bar I(\cR) = 1$: By positivity of $I_{\cR_{\min}}$, there is a unique $x_0 \in int(\supp\cR)$ such that either (i) $m^l_\cR(x_0) = 1$ or (ii) $v_\cR(x_0) \times \sigma_\cR(x_0) = 2$. In case (ii), one of $\sigma_\cR(x_0\pm \epsilon)$ is twice a primitive vector so $m^l_\cR(x_1)=2$ for some $x_1 \in \partial(\supp\cR)$. Either way, $m^l(\cR) = 3$ and we claim $\cM_1(\ab)$ is empty if $(\lambda,J)$ is close to $(\pi_\bI^*a,\bar J)$.
\item $\bar I(\cR) = -1$: $\cR$ is a local bigon at $x_0$. We claim $J$-holomorphic curves in $\cM_1(\ab)$ correspond to index 1 Morse flows of $f_{x_0}$ and, thus, exist in pairs.
\end{enumerate}
We deal with (c2) and (c3) (see Figure \ref{fig:diff-nonexamples}) in \S \ref{section:mbargument} and (c1) in the rest of \S \ref{section:main}. 

\subsection{A Morse-Bott argument}\label{section:mbargument}

\begin{definition}
Let $\bcR = (\bcP^0,\bcP^2)$, $\bcR^1 = (\bcP^0,\bcP^1)$ and $\bcR^2=(\bcP^1,\bcP^2)$ be three regions. The \defn{sharing multiplicity} between $\bcR^1$ and $\bcR^2$ at $x$ is
\[ m^s_{\bcR^1,\bcR^2}(x) := m_{\bcP^1}(x) + m^{triv}_{\bcR}(x) - m^{triv}_{\bcR^1}(x) - m^{triv}_{\bcR^2}(x) \]
where $m^{triv}_{\bcR}(x)$ denotes the number of local bigon factors of $\bcR$ at $x$. We also write $m^{triv}(\bcR) = \sum_x m^{triv}_\bcR(x)$ and $m^s(\bcR^1,\bcR^2) = \sum_x m^s_{\bcR^1,\bcR^2}(x)$.

\end{definition}
When $\bcP^0,\bcP^1$ and $\bcP^2$ are drawn with each starting at $0 \in \ft^2$, $m^s_{\bcR^1,\bcR^2}(x)$ is the number of edges of $\cP^1$ at $x$ ``sandwiched'' between non-local factors of $\bcR^1$ and $\bcR^2$.

\begin{lem}\label{lem:complexity}
For $\bcR,\bcR^1$ and $\bcR^2$ as above, $m_{\bcR^1,\bcR^2}^s \ge 0$. If $\bcR$ is indecomposable and non-local,
\[ \bar I(\bcR)- \sum \bar I(\bcR^i_k) = m^s(\bcR^1,\bcR^2) = \sum m^l(\bcR^i_k) - m^l(\bcR) \]
where each sum is over all non-local factors $\bcR^i_k$ of $\bcR^i$ for $i=1,2$. In particular, if $\bar I(\bcR) =0$ and $m^l(\bcR)=2$, then either $\bcR^1$ or $\bcR^2$ is local.
\end{lem}
\begin{proof}
The first assertion is obvious. For the second, use
\[ m^l_{\bcR} - m^{triv}_{\bcR} =  \sum m^l_{\bcR^i_k}- m^s_{\bcR^1,\bcR^2} \]
and
\[ \bar I(\bcR)  + m^{triv}(\bcR) = \sum \bar I(\bcR^i_k) + m^s(\bcR^1,\bcR^2) \]
which follow from $m^l_{\bcR} = m^l_{\bcR^1} + m^l_{\bcR^2} - m_{\bcP^1}$ and $\bar I(\bcR) = \bar I(\bcR^1) + \bar I(\bcR^2) + m(\bcP^1)$.
\end{proof}

This is the only section where the choice of $f_x$ and $\bar J$ in \S \ref{section:mb} plays a role, due to:
\begin{lem}\label{lem:thetaconstraint}\cite[Proposition 10.16]{t3}
If $\bar C$ is a $\bar J$-holomorphic curve with positive ends at $m_i^+$-fold covers of $\brho_{x_i^+}(\theta^+_i)$ and negative ends at $m_j^-$-fold covers of $\brho_{x_j^-}(\theta^-_j)$,
\[ \Theta(\bar C) := \sum  m_i^+\theta_i^+ - \sum m_j^-\theta_j^- =0 \in \bR/\bZ. \]
\end{lem}
\begin{proof}
If a surface $Z \subset T^2$ bounds $\brho_x(\theta)$ and a representative of $[\brho_x(\theta)]$ contained in $(S^1 \times \{0\}) \cup (\{0\} \times S^1)$, then $\int_Zdt_1dt_2  = \theta \in \bR/\bZ$. Moreover, when restricted to $(v,\bar Jv)$ for $v \in T(\bR \times Y)$, $dsdx$ agrees with $d\bar Rd\bar Q =  a \wedge a'/(a\times a') = dt_1dt_2$. Hence,
\[ \Theta(\bar C) = \int_{\bar C} dt_1dt_2 = \lim_{n \to \infty} \int_{\bar C_n}dt_1dt_2  = \lim_{n \to \infty}\int_{\bar C_n} dsdx = \lim_{n \to \infty}\int_{\partial \bar C_n}(-xds) = 0 \]
where $\bar C_n$ denotes $\bar C \cap ([-n,n] \times Y)$.
\end{proof}

\begin{prop}\label{prop:diff-mb}
Let $\lambda$ be a good perturbation of $\pi_\bI^*a$ and $J$ a generic $\lambda$-admissible almost complex structure. Let $\alpha$ and $\beta$ be orbit sets of $\lambda$ and suppose $\cR_\ab$ is indecomposable with $I(\cR_\ab)=1$ and $\bar I(\cR_\ab)=\pm 1$. Then, $\#(\cM_1(\alpha,\beta)/\bR) = 0$ for $(\lambda,J)$ sufficiently close to $(\pi_\bI^*a,\bar J)$.
\end{prop}

\begin{proof}
Consider a sequence $(\lambda_n,J_n)$ of good perturbations of $\pi_\bI^*a$ and generic $\lambda_n$-admissible $J_n$ converging to $(\pi_\bI^*a,\bar J)$. Suppose there exists a $J_n$-holomorphic curve $C_n \in \cM_1^{J_n}(\alpha,\beta)$ for each $n$. Each $C_n$ satisfies the partition condition and, by Proposition \ref{prop:indecomposable}, has genus zero. By \cite{mb}, a subsequence of $C_n$ converges to a $\bar J$-holomorphic building $\bar C = (\bar C^1, \cdots, \bar C^l)$ where $\bar C^i \in \cM^{\bar J}(\alpha^i,\beta^i)$ for orbit sets $\alpha^i$ and $\beta^i$ of $\pi_\bI^*a$ (see \S \ref{section:mb}). Let $\bcP^{i-1} := \bcP_{\alpha^i}$ and $\bcP^i:= \bcP_{\beta^i}$ for $1 \le i \le l$. By the definition of $\bar J$-holomorphic buildings, the two definitions of $\bcP^i$ agree for $1 \le i<l$ and $\cR_\ab$ decorates $(\bcP^0,\bcP^l)$. Write $\bar C^i = \bar C^i_0 \cup \bar T^i$ where $\bar T^i$ is the union of its (multiply covered) trivial cylinder components. By Lemma \ref{lem:apos}, each $\bcR^i=(\bcP^{i-1},\bcP^i)$ is $a$-positive and $\bar C^i= \bar T^i$ whenever $\bcR^i$ is local.

If $\bar I(\cR_\ab) =-1$, i.e. $\bar\cR_\ab$ is a local bigon at some $x_0$, then $\bar C$ is a trivial cylinder attached to Morse flows of $f_{x_0}$. By \cite{mb}, for sufficiently large $n$, $C_n$ is one of the two cylinders corresponding to the two flows from $\theta_{x_0}^{\max}$ to $\theta_{x_0}^{\min}$ and $\#(\cM_1^{J_n}(\alpha,\beta)/\bR) =0$. Henceforth assume $\bar I(\cR_\ab) = 1$.

\step{1.} Let $i_1$ and $i_2$ be the smallest and the largest $i$ such that $\bcR^i$ is non-local. If $i_1=i_2$, each positive (respectively negative) end of $\bar C^{i_1}$ converges to covers of some $\brho_x(\theta_x^{\min})$ (respectively $\brho_x(\theta_x^{\max})$). Since $m^l(\cR_\ab)=3$ (see \S \ref{section:classify} (c2)),
\[ \Theta(\bar C^{i_1}) = 3\theta_h + (m(\cR_\ab)-3)\theta_h/N \neq 0, \]
contradicting Lemma \ref{lem:thetaconstraint}. Thus, $i_1<i_2$. By Lemma \ref{lem:complexity} on $(\bcR')^1 :=(\bcP^0,\bcP^{i_1})$ and $(\bcR')^2 := (\bcP^{i_1},\bcP^l)$, $\sum m^l((\bcR')^j_k) + \sum \bar I((\bcR')^j_k) = 4$. Since each $(\bcR')^j$ is non-local and $a$-positive, each contains exactly one non-local factor with $m^l((\bcR')^j_k)=2$ and $\bar I((\bcR')^j_k) = 0$. By Lemma \ref{lem:complexity} on $(\bcP^{i_1},\bcP^{i_2-1})$ and $(\bcP^{i_2-1},\bcP^l)$, we conclude $\bcR^{i_1} =(\bcR')^1$, $\bcR^{i_2} = (\bcR')^2$, all other $\bcR^i$ are local and $m^s_{\bcR^{i_1},\bcR^{i_2}}=\delta_{x_0}$ for some $x_0$.

\step{2.} For $i_0=i_1$ or $i_2$, we claim $\bcR^{i_0}_0 = \bcR_{\alpha_0,\beta_0}$ where $\bcR^{i_0}_0$ denotes the non-local factor of $\bcR^{i_0}$ and $\bar C^{i_0}_0 \in \cM(\alpha_0,\beta_0)$. Write $\supp\bcR^{i_0}_0=[x_1,x_2]$ and note $\bar C^{i_0}_0 \cap (\bR \times \{x\} \times T^2) = \emptyset$ for $x < x_1$ by Lemma \ref{lem:apos} and for $x=x_1$ since $\bar C^{i_0}_0$ intersects any trivial cylinder at $x_1$ transversely. Hence, $\bar C^{i_0}_0 \cap (\{0\} \times [0,x_1+\epsilon] \times T^2)$ is empty for some $\epsilon >0$. Since a trivial cylinder at $x < x_1+\epsilon$ with $v_a(x) \times \sigma_{\bcR^{i_0}}(x) = 1$ intersects each component of $\bar C^{i_0}_0 \cap (\bR \times [0,x_1+\epsilon] \times T^2)$ at least once by \eqref{eq:pos}, $\bar C^{i_0}_0$ cannot have both a positive and a negative end at $x_1$, and similarly at $x_2$. The claim follows since $m^l(\bcR^{i_0}_0)=2$ and $C^{i_0}$ contains no trivial cylinders on $(x_1,x_2)$.

\step{3.} By Lemma \ref{lem:thetaconstraint}, $\bar C^{i_1}_0$ has a negative end at some $\brho_{x_1}^{m_1}(\theta^-)$ with $\theta^- \neq \theta_{x_1}^{\max}$. Since $C^{i_2}_0$ must have a positive end at $x_1$, $x_1 \not \in \partial(\supp\cR_\ab)$ by positivity of $\bcR^{i_1}$ and $\bcR^{i_2}$ at $x_1$. By Lemma \ref{lem:pos}(b), $\sigma_{\cR_\ab}(x_1) = \sigma_{\bcR^{i_1}}(x_1) + \sigma_{\bcR^{i_2}}(x_1)$ is not a multiple of $v_a(x_1)$, so $x_1 \in int(\supp\bcR^{i_0})$ for $i_0=i_1$ or $i_2$. In either case, $m_{\bcP^{i_1}}(x_1) > m^{triv}_{\bcR^{i_0'}}(x_1) = m_{\bcP^{i_1}}(x_1) - m^s_{\bcR^{i_1},\bcR^{i_2}}(x_1)$ for $i_0'=i_1+i_2-i_0$. Hence, $x_1 = x_0$ and the total multiplicity of trivial cylinders of $\bar C^{i_0'}$ at $x_0$ is $m_{\bcP^{i_1}}(x_0)-1$ by Step 2. We conclude $m_1=1$ and any other negative end of $\bar C^{i_1}$ is at some $\brho_x^m(\theta_x^{\max})$. Similarly, $\bar C^{i_2}$ has a single positive end at $\brho_{x_0}(\theta^+)$ with $\theta^+ \neq \theta_{x_0}^{\min}$ and any other positive end at some $\brho_x^m(\theta_x^{\min})$.

\step{4.} Suppose $\brho_{x_0}$ is convex. By Step 2, 3 and $m^l(\bcR^{i_1}_0) =2$,
\[ \Theta(\bar C^{i_1}_0) = 2\theta_h  - \theta^- + (m(\bcR^{i_1}_0)-3)\theta_h/N =0. \]
Hence, $2\theta_h \le \theta^- < 3\theta_h$. Similarly, $-2\theta_h < \theta^+ \le -\theta_h$ by
\[ \Theta(\bar C^{i_2}_0) = \theta_h + \theta^+ + (m(\bcR^{i_2}_0)-2)\theta_h/N = 0. \]
However, there is no flow of $f_{x_0}$ from $\theta^-$ to $\theta^+$. If $\brho_{x_0}$ is concave, we have $\theta_h < \theta^- <2\theta_h$ and $-3\theta_h < \theta^+ \le -2\theta_h$ and arrive at a similar contradiction.
\end{proof}

\subsection{Invariance of the moduli count}

We establish invariance of moduli count under certain deformations of $(\lambda,J)$.

\begin{definition}
Let $\bcR$ be a non-local indecomposable region with $\supp(\bcR) = [x_1,x_2]$. We say that an orbital moment map $a$ is \defn{$\bcR$-adapted} if $\bcR$ is $a$-positive, $a(x) \times a'(x_1) > 0$ for $x < x_1$ and $a(x) \times a'(x_2) > 0$ for $x > x_2$.
\end{definition}

\begin{lem}\label{lem:lambdaconnected}
For a non-local indecomposable region $\bcR$, the space of $\bcR$-adapted orbital moment maps is path-connected.
\end{lem}

\begin{proof}
Observe that a \defn{rescaling} $\kappa a$ of an orbital moment map $a$ by $\kappa: \bI \to \bR^+$ and a \defn{reparametrization} $\psi^*a$ by $\psi \in \aut(\bI)$ are connected to $a$ via linear interpolations of $\kappa$ with $\kappa' \equiv 1$ and $\psi$ with $\id$, respectively. 

Let $a$ be any $\bcR$-adapted orbital moment map and let $P=\{x_1 < \cdots < x_n\}$ be $\supp(m_\bcR)$ if $m(\bcR)>2$; and $\supp(m_\bcR) \cup \{x_2\}$ for $x_2 \in int(\supp\bcR)$ otherwise. In either case, we can find $u_i$ near $-\sigma_\bcR(x_{i+1})^\vee$ so that $u_i \times a' > 0$ on $[x_i,x_{i+1}]$ for each $1 \le i < n$. We claim there exist $\{z_i\}_{i=1}^{n-1}$ and a continuous rescaling $\bar b$ of $a$ such that $z_i \in (x_i,x_{i+1})$, $\bar b$ is smooth except (possibly) at $z_i$ and $\bar b'(x_i) \sim a'(x_i)$ for each $i$: if $u_i \times a(x_i) \ge 0$, pick $z_i$ close to $x_i$; if $u_i \times a(x_{i+1}) \le 0$, pick $z_i$ close to $x_{i+1}$; if neither, pick $z_i$ so that $a(z_i) \sim u_i$. The claim follows from $a \times a' > 0$ and $u_i \times a' > 0$. Obtain a smooth rescaling $b$ of $\bar b$ by modifying $\bar b$ on small neighbourhoods of $z_i$ so that $b \times b'$ does not change signs on $[x_i,x_{i+1}]$.

For $\bcR$-adapted $a_0$ and $a_1$, find $b_0$ and $b_1$ using the above procedure. By rescaling, assume $a_0$ and $a_1$ are $C^1$-close to $b_0$ and $b_1$ and by reparametrizing, $b_0'(x) \sim b_1'(x)$ for all $x$. Then, $a_r:=(1-r) a_0 + r a_1$ is an orbital moment map since $b_r \times b_r' > 0$ for $b_r:=(1-r)b_0+rb_1$, and it is $\bcR$-adapted since $a_0$ and $a_1$ are.
\end{proof}

The following is an adaptation of \cite[Lemma 3.15]{pfh}:

\begin{prop}\label{prop:deform}
Let $\cR$ be an indecomposable region with $I(\cR)=1$ and $\bar I(\cR) \le 0$. For each $r \in \{0,1\}$, let $a_r$ be an $\cR$-adapted orbital moment map, $\lambda_r$ a good perturbation of $\pi_\bI^*a_r$ and $J_r$ a generic $\lambda_r$-admissible almost complex structure. Then
\[ \#(\cM^{J_0}_1(\alpha_0,\beta_0)/\bR) = \#(\cM^{J_1}_1(\alpha_1,\beta_1)/\bR) \]
where $\alpha_r$ and $\beta_r$ are orbit sets of $\lambda_r$ with $\cR_{\alpha_r,\beta_r} = \cR$.
\end{prop}

\begin{proof}
Suppose we have paths $\{a_r\}_{r \in [0,1]}$ of $\cR$-adapted orbital moment maps with $\cA^r:=\sum_x m_\cR(x)\cA_{a_r}(x)$; and $\{\lambda_r\}$ of perturbations of $\pi_\bI^*a_r$ so that, for each $r \in [0,1]$ and $x \in \supp(m_\cR)$, $\lambda_r|_{U_x^r \times T^2}$ is of the form \eqref{eq:mblambda} for a neighbourhood $U_x^r$ of $x$ with $\cA_{a_r}|_{U_x^r \setminus \{x\}}> \cA^r$. Let $\alpha_r$ and $\beta_r$ denote the orbit sets of $\lambda_r$ with $\cR_{\alpha_r,\beta_r}=\cR$ and choose a generic path $\{J_r\}$ of $\lambda_r$-admissible almost complex structures. Assuming compactness of $\cM_1^{J_r}(\alpha_r,\beta_r)/\bR$ as in \cite{bn}, its mod 2 count can change only when there is a broken $J_r$-holomorphic curve $C=(C^1,\cdots, C^l)$ from $\alpha_r$ to $\beta_r$.

Due to action, a $d$-fold cover $dC'$ of a somewhere injective $J_r$-holomorphic curve $C' \subset \bR \times U_x^r \times T^2$ has $I(dC') = dI(C') \ge 0$ with equality only if $C'$ is trivial. This and Lemma \ref{lem:pos} implies each $I(C^i) \ge 0$ and some $C^i$ must contain a somewhere injective component $C' \subset \bR \times U_x^r \times T^2$ with $I(C') = 1$. If $\cR$ is local, $C=C'$ and we are done. Otherwise, by Proposition \ref{prop:diff-mb} and the local case of this Proposition, such $C'$ exists in pairs and by standard gluing results as in \cite{lee}, $C$ also exists in pairs (cf. automatic transversality \cite{wendl}). It remains to provide such paths and justify compactness. 

Write $\lambda_r = (1+ \eta_r\sum \tilde f^r_{x^r})\pi_\bI^*a_r$ for $r=0,1$. If $a_0=a_1$, linearly interpolating $\sum\tilde f^r_{x^r}$ and choosing sufficiently small $\eta_r$ gives a path of good perturbations of $\pi_\bI^*a_0$ and Gromov compactness holds. In general, let $\{a_r\}$ be given by a linear interpolation if $\cR$ is local, and by Lemma \ref{lem:lambdaconnected} otherwise. For each $x \in P:=\supp(m_\cR)$, the first step allows us to assume $\tilde f^0_x=\tilde f^1_x$ and $U_x^r:=U_x^0$ is small enough so that $\lambda_r := (1+ \sum_{x \in P} \eta_r \tilde f^0_x) \pi_\bI^*a_r$ gives a desired path for small enough $\eta_r$.

Finally, we justify Gromov compactness when $\cR$ is not local. Let $C \in \cM_1^{J_r}(\alpha_r,\beta_r)$ and $x_0 \in P$. We claim $\pi|_C^{-1}(U_{x_0}^0)$ is connected where $\pi:\bR \times Y \to \bI$ is the obvious projection. For each of its component $C'$, $v_\cR(x_0) \times [\pi|_{C'}^{-1}(x)] \ge 0$ for generic $x \in U_{x_0}^0$ by \eqref{eq:pos}. If this is equality for some $C'$, it has a positive end at a convex orbit or a negative end at a concave orbit (cf. proof of Lemma \ref{lem:pos}(b)) so $x_0 \in \partial(\supp\cR)$ and the claim follows from $m_\cR(x_0)=1$. Otherwise, it follows from $v_\cR(x_0) \times \sigma_\cR(x) \le 1$. In turn, $\pi|_C^{-1}(V)$ for any component $V$ of $\supp(\cR) \setminus \cup_{x \in P}U_x^0$ is a single finite cylinder since $C$ does not have genus by \eqref{eq:fredind} and any component of $\pi|_C^{-1}(\partial V)$ is homologically non-trivial by \eqref{eq:pos}. Since $\lambda_r$ does not have an orbit $\rho$ in $V$ with $[\rho] = \pm \sigma_\cR(x)$ for $x \in V$, a sequence $C_n \in \cM_1^{J_{r_n}}(\alpha_{r_n},\beta_{r_n})$ cannot develop an end in $V$. (When there is a birth/death of such $\rho$, this assertion fails and a bifurcation does occur.)
\end{proof}

\subsection{Base cases for induction}\label{section:base}

Let $\lambda$ be a good perturbation of $\pi_\bI^*a$ for a generic orbital moment map $a$ and $J$ a generic $\lambda$-admissible almost complex structure.
\begin{prop}\label{prop:base}
Let $\alpha$ and $\beta$ be orbit sets of $\lambda$. Suppose $\cR_\ab$ is $a$-positive, $I(\cR_\ab)=1$ and one of the following holds:
\begin{enumerate}[(i)]
\item $m_{\cP_\alpha}=\delta_{x_+}$ and $m_{\cP_\beta} = \delta_{x_-}$ for $\delta_+ \neq \delta_-$;
\item $c_\cR \ge 0$, $m_{\cP_\alpha}=\delta_{x_1} + \delta_{x_2}$, $m_{\cP_\beta}=m\delta_{x_0}$ for $x_0 \in (x_1,x_2)$, and $m^h_{\cP_\beta} (x_0)= 0$ if $m > 1$; or
\item $c_\cR \le 0$, $m_{\cP_\beta}=\delta_{x_1} + \delta_{x_2}$, $m_{\cP_\alpha}=m\delta_{x_0}$ for $x_0 \in (x_1,x_2)$, and $m^h_{\cP_\alpha} (x_0)= 0$ if $m > 1$.
\end{enumerate}
Then $\#(\cM_1(\alpha,\beta)/\bR) =1$.
\end{prop}

\begin{proof}
Since $\bar I(\cR_\ab) =0$, we may assume $a$ is any $\cR_\ab$-adapted orbital moment map and $\lambda$ any good perturbation by Proposition \ref{prop:deform}. With this assumption, cases (ii) and (iii) are covered in \cite{pfh,t3} but we refer to their original source in \cite{taubes2}.

\case{(i)} Assume $x_\pm$ are critical points of $g(x) := \pm x(x-1/2)(x-1)$ with a plus sign if $x_+< x_-$ and minus otherwise. By $SL_2(\bZ)$-symmetry, assume $v(x_+) = (1,1)$ and $\lambda = (1+\eta f)(1+\eta' g)\pi_\bI^* a_{\lambda_{std}}$ for a small $\eta' >0$ where $a_{\lambda_{std}}(x) = (1-x,x)$ is the orbital moment map of $(S^3,\lambda_{std})$ and $f := \sum \tilde f_x$ in \eqref{eq:mblambda}. Finite dimensional Morse-Bott theory on $S^2$-family of embedded orbits of $\lambda_{std}$ gives a unique flow of $(1+\eta f)(1+\eta' g)$ from $\brho_{x_+}(\theta_e)$ to $\brho_{x_-}(\theta_h)$. By a Morse-Bott argument \cite{mb}, the $J$-holomorphic cylinder from $\check e_{x_+}$ to $\hat h_{x_-}$ corresponds to this flow for small $\eta$ and $\eta'$. We similarly get the unique member of $\cM_1(\check h_{x_+},\hat e_{x_-})/\bR$.

\case{(ii)} There is an identification of $\bR \times (\bI \times S^1 \times S^1)$ with a subset of $(\bR \times S^2 \times S^1,\alpha, J)$ considered in \cite{taubes2} so that the pullback of $\alpha$ is $\pi_\bI^*a$ for an everywhere convex $\cR_\ab$-adapted $a$ and the pullback of $J$ is $\pi_\bI^*a$-admissible. The unique member of $\cM_1(\alpha,\beta)/\bR$ for $m=0$ corresponds to an $(\bR \times S^1)$-family of $J$-holomorphic cylinders in \cite[Theorem A.1(c)]{taubes2} via a Morse-Bott argument \cite{mb}. If $m>0$, $C \in \cM_1(\alpha,\beta)$ has one negative puncture by partition condition and we similarly get the unique member of $\cM_1(\alpha,\beta)/\bR$ from an $(\bR \times S^1 \times S^1)$-family of three-punctured $J$-holomorphic spheres in \cite[Theorem A.2]{taubes2}.

\case{(iii)} We reduce to case (ii) with $m>0$. Assume $a$ is $\cR_\ab$-adapted and everywhere concave and $rdt_1 \in \im(a)$ for some $r \in \bR$. Let $\psi(t_1,t_2) = (t_1,-t_2)$ and $\Psi = (-\id_\bR) \times \id_\bI \times \psi$ be diffeomorphisms of $T^2$ and $\bR \times \bI \times T^2$. Then $\lambda^\dagger=2rdt_1 - (\id_\bI \times \psi)^*\lambda$ is a good perturbation of an everywhere convex orbital moment map and, for each orbit set $\gamma$ of $\lambda$, $\gamma^\dagger = (\id_\bI\times\psi)_*^{-1}\gamma$ is an orbit set of $\lambda^\dagger$. We have $\cM_1^J(\alpha,\beta)=\cM_1^{J^\dagger}(\beta^\dagger,\alpha^\dagger)$ where $J$ is a $\lambda$-admissible almost complex structure and $J^\dagger := \Psi_*^{-1} J\Psi_*$. Furthermore, if $J$ is close to $\pi_\bI^*a$-admissible $\bar J$ with $\bar J(\partial_x) = c\bar Q$, then $J^\dagger$ maps $\partial_s$ to a positive multiple of the Reeb vector field of $\lambda^\dagger$ and is $d(e^s\lambda^\dagger)$-tame provided $c>0$ is small. We can deform $J^\dagger$ to a $\lambda^\dagger$-admissible almost complex structure without changing the moduli count, similarly to Proposition \ref{prop:deform}.
\end{proof}

\subsection{Induction step}

\begin{prop}\label{prop:suff}
Let $(Y,\lambda,J)$ be as in \S \ref{section:base} and let $\alpha$ and $\beta$ be orbit sets of $\lambda$. Suppose $\cR_\ab$ is indecomposable and $a$-positive with $I(\cR_\ab) = 1$ and $\bar I(\cR_\ab) = 0$. Then $\#(\cM_1(\alpha,\beta)/\bR) = 1$.
\end{prop}

\begin{proof}
Since $\cR_\ab$ is minimally positive, write $\supp\cP_\alpha \sqcup \supp\cP_\beta = \{x_1 < \cdots < x_n\}$. We define the ``induction complexity'' $z(\cR_\ab)$ as follows: if $\cR_\ab$ is of the form in Proposition \ref{prop:base}, set $z(\cR_\ab) := 1$. If $n=3$ with $m_{\cP_\alpha}(x_1)=m_{\cP_\alpha}(x_3)=m^e_{\cP_\beta}(x_2) = m^h_{\cP_\beta}(x_2) = 1$ or $m_{\cP_\beta}(x_1)=m_{\cP_\beta}(x_3)=m^e_{\cP_\alpha}(x_2) = m^h_{\cP_\alpha}(x_2) = 1$,  set $z(\cR_\ab) :=2$. Otherwise, $z(\cR_\ab) :=m(\cR_\ab) \ge 3$. We induct on $z$.

Assume $x_1 \in \supp\cP_\alpha$ as the case with $x_1 \in \supp\cP_\beta$ can be argued similarly by reversing the roles of $\cP_\alpha$ and $\cP_\beta$. Further assume $m_{\cP_\alpha}^h(x_1)=1$ as the case with $m_{\cP_\alpha}^e(x_1)=1$ can be argued similarly by switching the roles of $x_i$ and $x_{n-i}$. Note that $\bar I(\cR_\ab)=0$ ensures admissibility of $\alpha$ and $\beta$ and that $Q_{\cR_\ab}(x_2) \neq 0$ since we assume $z(\cR_\ab) > 1$.

\begin{figure}
\begin{center}
\begin{tikzpicture}[scale=1.4]
\draw[->,color=red] (0,0) -- (-1,1) node[midway,above right] {$\check h_{3/8}$} -- (-2,.8) node[left] {$\cP^+$};
\draw[->,color=blue] (0,0) -- (-1,0) node[pos=.6,below] {$\check h_{1/2}$} -- (-2,-.2) node[left] {$\cP^-$};
\draw[->] (-1,-1) -- (0,0) node[midway,below right] {$\check e_{1/8}$};
\draw[->] (-1,-1) -- (-1,0) node[midway,left] {$\check e_{1/4}$};
\draw[dashed][->] (-1,0) -- (-1,1) node[midway,left] {$\check e_{1/4}$};
\begin{scope}[xshift=25ex]
\draw[->,color=red] (0,0) -- (-1,1) node[midway,above right] {$\check h_{3/8}$} -- (-2,.8) node[left] {$\cP^+$};
\draw[->,color=blue] (0,0) -- (-1,0) node[pos=.6,below] {$\check e_{1/2}$} -- (-2,-.2) node[left] {$\cP^-$};
\draw[->] (-1,-1) -- (0,0) node[midway,below right] {$\check h_{1/8}$};
\draw[->] (-1,-1) -- (-1,0) node[midway,left] {$\check e_{1/4}$};
\draw[dashed][->] (-1,0) -- (-1,1) node[midway,left] {$\check h_{1/4}$};
\end{scope}
\end{tikzpicture}
\end{center}
\vspace{-.2in}
\caption{Case 1 of the induction step.}
\label{fig:ind1}
\end{figure}
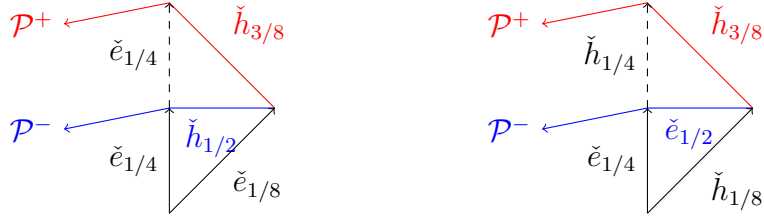

\case{1.} If $x_2 \in \supp\cP_\beta$, assume $v_{\cP_\alpha}(x_1) = (-1,1)$ and $v_{\cP_\beta}(x_2) = (-1,0)$ by $SL_2(\bZ)$-symmetry and assume $(a')^\vee|_{[0,1/2 + \epsilon)} \sim (\cos 2\pi x, \sin 2\pi x)$ using Proposition \ref{prop:deform} and $\aut(\bI)$. Write $\cP_\alpha = \cP_{\check h_{3/8}}\cP^+$ and $\cP_\beta = \cP_{1/2}\cP^-$ where $\cP_{1/2}=\cP_{\check h_{1/2}}$ if $m^h_{\cP_\beta}(1/2) > 0$ and $\cP_{\check e_{1/2}}$ otherwise. Let $\tilde\alpha$ and $\tilde\beta$ be orbit sets so that $\cP_{\tilde\alpha}= \cP_{1/8}\cP_\alpha$ and $\cP_{\tilde\beta} = \cP_{\check e_{1/4}}\cP^-$ where $\cP_{1/8} = \cP_{\check e_{1/8}}$ if $m^h_{\cP_\beta}(1/2) > 0$ and $\cP_{\check h_{1/8}}$ otherwise. See Figure \ref{fig:ind1}. We consider $\gamma$ for which $\<\partial\tilde\alpha,\gamma\>\<\partial\gamma,\tilde\beta\> \neq 0$.

Each of $\cR^1=(\cP_{\tilde\alpha},\cP_\gamma)$ and $\cR^2=(\cP_\gamma,\cP_{\tilde\beta})$ contains exactly one non-local factor $\cR^i_0$ and $\bar I(\cR^i_0) = 0$ by Lemma \ref{lem:complexity} since $\bar I(\cP_{\tilde\alpha},\cP_{\tilde\beta}) = 1$. By positivity of $\cR^i$, $\min(\supp\cP_\gamma) = 1/8$ or $1/4$. In the first case, write $\supp(\cR^2_0) = [1/8,x_0]$ where $x_0 \ge 1/2$ since $\cP_{\tilde\beta}|_{(1/4,1/2)}\equiv 0$. By minimal positivity of $\cR^2_0$, $\cP_\gamma|_{(1/8,1/2)} \equiv 0$ and by $a$-positivity of $\cR^2_0$ at $x=1/2$, $x_0$ must be $1/2$. Since $I(\cR^2_0) =1$, we have $\cP_\gamma = \cP_{1/8}\cP_\beta$ and $\<\partial \tilde\alpha,\gamma\>\<\partial \gamma, \tilde\beta\> =  \<\partial\alpha,\beta\>$ by induction hypothesis. In the latter case, $\supp(\cR^1_0) = [1/8, 3/8]$ by minimal positivity of $\cR^1_0$ and $\cP_\gamma= \cP_{\check e_{1/4}}\cP_{\check e_{1/4}}\cP^+$ or $\cP_{\check e_{1/4}}\cP_{\check h_{1/4}}\cP^+$, whichever makes $I(\cR^1_0)=1$. We verify that $z(\cR^i_0) < z(\cR_\ab)$ for each $i$: if $z(\cR_\ab) = 2$, $z(\cR^i_0)= 1$; otherwise, it follows from $m(\cR^2_0) = m(\cR_\ab) - 1$. By induction hypothesis and $\partial^2 = 0$, we conclude $\<\partial\alpha,\beta\> = 1$.

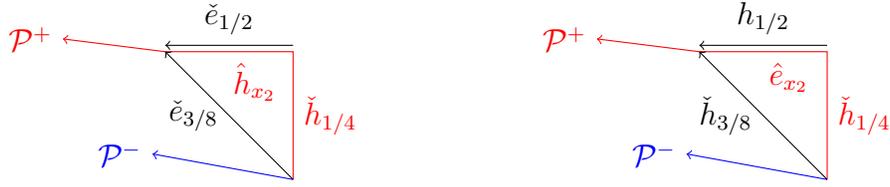
\begin{figure}
\begin{center}
\begin{tikzpicture}[scale=1.7]
\draw[->,color=red] (0,0) -- (0,1) node[midway,right] {$\check h_{1/4}$} -- (-1,1) node[pos=.3,below] {$\hat h_{x_2}$} -- (-1.8,1.1) node[left] {$\cP^+$};
\draw[->,color=blue] (0,0) -- (-1.1,.2) node[left] {$\cP^-$};
\draw[->] (0,0) -- (-1,1) node[midway,left] {$\check e_{3/8}$};
\draw[->] (0,1.05) -- (-1,1.05) node[midway,above] {$\check e_{1/2}$};
\begin{scope}[xshift=23ex]
\draw[->,color=red] (0,0) -- (0,1) node[midway,right] {$\check h_{1/4}$} -- (-1,1) node[pos=.3,below] {$\hat e_{x_2}$} -- (-1.8,1.1) node[left] {$\cP^+$};
\draw[->,color=blue] (0,0) -- (-1.1,.2) node[left] {$\cP^-$};
\draw[->] (0,0) -- (-1,1) node[midway,left] {$\check h_{3/8}$};
\draw[->] (0,1.05) -- (-1,1.05) node[midway,above] {$\check h_{1/2}$};
\end{scope}
\end{tikzpicture}
\end{center}
\vspace{-.1in}
\caption{Case 2 of the induction step.}
\label{fig:ind2}
\end{figure}

\case{2.} If $x_2 \in \supp\cP_\alpha$, assume $v_{\cP_\alpha}(x_1) = (0,1)$, $v_{\cP_\alpha}(x_2) = (-1,0)$ for $x_2=1/2+2\epsilon$ and $(a')^\vee|_{[0,1/2+\epsilon]} \sim (\cos 2\pi x, \sin 2\pi x)$ with $\cA_a|_{(1/2,x_2)} > L$. Write $\cP_\alpha = \cP_{\check h_{1/4}}\cP_{x_2}\cP^+$ where $\cP_{x_2}=\cP_{\hat h_{x_2}}$ if $m^h_{\cP_\alpha}(x_2) > 0$ and $\cP_{\hat e_{x_2}}$ otherwise. Let $\tilde\alpha$ be an orbit set so that $\cP_{\tilde\alpha}= \cP_{\check h_{1/4}}\cP_{1/2}\cP^+$ where $\cP_{1/2} = \cP_{\check e_{1/2}}$ if $m^h_{\cP_\alpha}(x_2) > 0$ and $\cP_{\check h_{1/2}}$ otherwise. See Figure \ref{fig:ind2}.

As before, if $\<\partial\tilde\alpha,\gamma\>\<\partial\gamma,\beta\> \neq 0$, each of $\cR^1=(\cP_{\tilde\alpha},\cP_\gamma)$ and $\cR^2 = (\cP_\gamma,\cP_\beta)$ contains exactly one non-local factor $\cR^i_0$ and $\bar I(\cR^i_0) = 0$. If $\min(\supp \cP_\gamma) = 1/4$, $m_{\cP_\gamma}(1/2) = 0$ by minimal positivity of $\cR^2_0$ and $m_{\cP_\gamma}(x_2) = 1$ by $a$-positivity of $\cR^1_0$, yielding $\<\partial\tilde\alpha,\gamma\>\<\partial\gamma,\beta\> = \<\partial\alpha,\beta\>$ by induction hypothesis. If $\min(\supp \cP_\gamma) > 1/4$, $\supp(\cR^1_0) = [1/4,1/2]$, so $\cP_\gamma = \cP_{\check e_{3/8}}\cP^+$ or $\cP_{\check h_{3/8}}\cP^+$, whichever makes $I(\cR^1_0)=1$. Since $z(\cR^i_0) < z(\cR_\ab)$, $\<\partial\alpha,\beta\> = 1$ by induction hypothesis and $\partial^2 =0$.
\end{proof}

\section{ECC of closed manifolds}\label{section:closed}

\subsection{Toric contact $T^3$}\label{section:t3}

Define $s_n:\bR \to \bR$ by $s_n(x):=x-n$ and let $q:\bR \to S^1=\bR/\bZ$ be the quotient map. Previous definitions regarding lattice paths and regions still make sense when we replace $\bI$ with $\bR$ or $S^1=\bR/\bZ$ but we need to supplement Definition \ref{def:region} with:
\begin{definition}
An \defn{offset (lattice) region} $\bcR$ is a triple $(\bcP^0,\bcP^1,\sigma_0)$ where $\bcP^i:S^1 \to \bar\cV$ are paths with $[\bcP^0] = [\bcP^1]$ and $\sigma_0 \in \Lambda$. The \defn{slice class} of $\bcR$ at $x_0 \in [0,1) \overset{set}{=} S^1$ is
\[  \sigma_\bcR(x_0) := \sigma_0 - \sum_{x \in [0, x_0)}  m_{\bcP^0}(x) \cdot v_{\bcP^0}(x) + \sum_{x \in [0,x_0)}  m_{\bcP^1}(x) \cdot v_{\bcP^1}(x) \in \Lambda. \]
We say $\bcR_\bR = (\bcP_\bR^0,\bcP_\bR^1)$ with $\bcP_\bR^i: \bR \to \bar\cV$ is a \defn{lift} of $\bcR$ if $\cup_{n\in \bZ} s_n^*(\bcP_\bR^i) =q^*\bcP^i$ and $\sum_{n \in \bZ}\sigma_{\bcR_\bR}(n) = \sigma_0$. We similarly define a lift of a decorated $\cR$. A \defn{lift} of a pair $(\cP^0,\cP^1)$ is a lift of $(\cP^0,\cP^1,\sigma)$ for some $\sigma \in \Lambda$. We say $\bcR$ is \defn{decomposable} if it lifts to a decomposable region.
\end{definition}

In this section, a region will always mean an offset region unless we refer to a lift, in which case we use the subscript $\bR$. A concatenation $\bcR_1\bcR_2$ still makes sense if $int(\supp\bcR_i)$ are disjoint, and so does factoring $\bcR$ into indecomposable $\bcR_k$ up to cyclic ordering.

Now consider a generic orbital moment map $a:S^1 \to (\ft^2)^*$. Define $\varphi_a,\varphi_{a'}:\bR \to \bR$ so that $a^\vee \sim (\cos\varphi_a,\sin\varphi_a)$, $(a')^\vee \sim (\cos\varphi_{a'},\sin\varphi_{a'})$, $\varphi_a(0) \in [0,2\pi)$ and $\varphi_{a'}-\varphi_a \in (0,\pi)$. By $a \times a' > 0$, $\varphi_a(1)=\varphi_a(0) + 2\pi n_a$ for $n_a \ge 1$ and $\varphi_{a'}(1)=\varphi_{a'}(0) + 2\pi n_a$. Equip $Y = S^1 \times T^2$ with a small (in the sense of Proposition \ref{prop:diff-mb}) good perturbation $\lambda$ of $\pi_{S^1}^*a$ and $\bR \times Y$ with a generic $\lambda$-admissible almost complex structure $J$. For orbit sets $\alpha$ and $\beta$ of $\lambda$, define
\[ H_2(Y,\alpha,\beta,\sigma) := H_2(Y,\alpha,\beta,0) + [S^1]\times \sigma\]
for $H_2(Y,\alpha,\beta,0) := (q \times \id_{T^2})_*H_2([0,1)\times T^2,\alpha,\beta)$ and $\sigma \in H_1(\{pt\} \times T^2)$,
\[ \cM(\ab,\sigma) := \{\cC \in \cM(\ab)| [\cC] \in H_2(Y,\ab,\sigma)\} \]
and $\cM_k(\ab,\sigma):=\{\cC \in \cM(\ab,\sigma)|I(\cC)=k\}$.

\begin{prop}\label{prop:t3}
We adapt some previous results to this setting. The proofs remain nearly identical and uses the same trivialization $\tau$ of $\xi \isom \Span\{\partial_x, \bar Q\}$.
\begin{enumerate}[(a)]
\item Proposition \ref{prop:index} now asserts $I(\cR_{\ab,\sigma}) = I(\alpha,\beta,Z)$ for $Z \in H_2(Y,\alpha,\beta,\sigma)$: if $\cR_\bR$ is a lift of $\cR_{\ab,0}$ with $\supp(\cR_\bR) \subset [0,1)$, $I(\cR_{\ab,0})=I(\cR_\bR)$. Otherwise, use \eqref{eq:indamb} and
\[ I(\cP^0,\cP^1,\sigma_1) - I(\cP^0,\cP^1,\sigma_0) = 2[\cP^0] \times (\sigma_1-\sigma_0).\]
\item Lemma \ref{lem:pos} (a) works unmodified while part (b) says: if $\bcR$ is $a$-positive, then it is positive and, whenever $a'(x)^\vee \times \sigma_\bcR(x)=0$, it lifts to $q^*a$-positive $\bcR_\bR$ with $\supp(\bcR_\bR) \subset [x,x+1]$. Lemma \ref{lem:apos} says: $\cR_{\ab,\sigma}$ is $a$-positive if $\cM(\ab,\sigma) \neq \emptyset$, and the second assertion still holds.
\item  Lemma \ref{lem:complexity} works for $\bcR=(\bcP^0,\bcP^2,\sigma^1+\sigma^2), \bcR^1=(\bcP^0,\bcP^1,\sigma^1)$ and $\bcR^2=(\bcP^1,\bcP^2,\sigma^2)$. Lemma \ref{lem:thetaconstraint} works unmodified.
\end{enumerate}
\end{prop}

\begin{definition}\label{def:relevantlift}
We say a lift $\cR_\bR$ is \defn{relevant} if it satisfies the criteria in Theorem \ref{thm:main} {\em and} the normalizing conditions (i) $int(\supp \cR_\bR)= (x_1,x_2)$ for $x_1 \in [0,1)$ and $x_2 \in [x_1,x_1+1]$ and (ii) $\supp(m_{\cR_\bR}^{triv}) \subset [x_1,x_1+1)$.
\end{definition}

\begin{thm}\label{thm:t3}
Consider $(Y,\lambda,J)$ as above. For admissible orbit sets $\alpha$ and $\beta$ of $\lambda$, $\<\partial\alpha,\beta\> =1$ if and only if $(\cP_\alpha,\cP_\beta)$ admits a unique relevant lift.
\end{thm}

\begin{remark}\label{rem:uniquelift}(Uniqueness of a relevant lift)
Relevant lifts are possible only if we can write $\alpha = \alpha' \cup \gamma$ and $\beta = \beta' \cup \gamma$ with $m^l(\cR_{\alpha',\beta',0})=2$. Suppose $\cR_\bR$ and $\cR_\bR'$ are relevant lifts of $\cR_{\ab,\sigma}$ and $\cR_{\ab,\sigma'}$ with $\sigma \neq \sigma'$. Then, without loss of generality, $int(\supp\cR_\bR) = (x_1,x_2)$, $int(\supp\cR_\bR') = (x_2,x_1+1)$ and $m^l_{\cR_{\alpha',\beta',0}}= \delta_{x_1}+\delta_{x_2}$. Since $\cR_\bR$ is minimally positive, $\gamma$ has no orbits on $(x_1,x_2)$, and since $\sigma_{\cR_\bR'}|_{(x_1,x_2)} \equiv 0$, neither does $\alpha'\cup\beta'$. Similarly, $\alpha \cup \beta$ has no orbits on $(x_2,x_1+1)$. We conclude (i) $\supp(\cP_\alpha \cup \cP_\beta) = \{x_1,x_2\}$ and $a'(x) \times a'(x_1) \neq 0$ except at $x_1$ and $x_2$; and (ii) since $n_a \ge 1$, $a'$ is convex at $x_1$ and $x_2$ and $\beta'=\emptyset$. In all other cases, we have at most one $\sigma$ for which $\cR_{\ab,\sigma}$ admits a relevant lift. Moreover, this lift is unique unless $m^l_{\cR_{\alpha',\beta',0}} = 2\delta_{x_0}$: in this case, there are two relevant lifts (with $\supp(\cR_\bR)=[x_0,x_0+1]$), one with a hyperbolic edge at $x_0$ and the other at $x_0+1$.
\end{remark}

We first show a basic property of (offset) regions:
\begin{lem}\label{lem:nokink}
Let $\bcR=(\bcP^0,\bcP^1,\sigma)$ be $a$-positive. If $\bcP^0|_{(x_1,x_2)} \equiv 0$ and $\sigma_\bcR$ does not vanish on $(x_1,x_2)$, then $\varphi_{a'}(x_2)-\varphi_{a'}(x_1) \le \pi$ with equality only if $Q_\bcR = 0$ at $x_1$ and $x_2$. In particular, if $\bcR$ does not lift, $m(\bcP^0) > 2$.
\end{lem}

\begin{proof}
By $a$-positivity, we can choose $\varphi_\sigma:(x_1,x_2) \to \bR$ so that $\sigma_\bcR \sim (\cos \varphi_\sigma,\sin\varphi_\sigma)$ and $\varphi_\sigma - \varphi_{a'}\in (0,\pi)$. The first statement follows since $\varphi_\sigma$ jumps by $\vartheta \in (-\pi,0)$ at $\supp\bcP^1$ and is constant everywhere else. The second follows from $n_a \ge 1$.
\end{proof}

\begin{proof}[Proof of Theorem \ref{thm:t3}]
Suppose $m^l(\cR_{\ab,0}) \ge 2$. Equip $\tilde Y = \bR \times T^2$ with a contact form $\tilde \lambda := (q \times \id_{T^2})^*\lambda$ and $\bR \times \tilde Y$ with an almost complex structure $\tilde J := \tilde q_*^{-1}J \tilde q_*$ where $\tilde q = \id_\bR \times q \times \id_{T^2}$. We claim $\cM_1(\alpha,\beta) = \cup \cM_1^{\tilde J}(\tilde\alpha,\tilde\beta)$ where the union is over orbit sets $\tilde\alpha$ and $\tilde\beta$ of $\tilde\lambda$ for which $\cR_{\tilde\alpha,\tilde\beta}$ is a relevant lift of $(\cP_\alpha,\cP_\beta)$. If $C \in \cM_1(\ab)$, the non-trivial component of $C$ has genus zero by \eqref{eq:fredind} so there is $\tilde C \in \cM_1^{\tilde J}(\tilde\alpha,\tilde\beta)$ where $\tilde\alpha$ and $\tilde\beta$ are orbit sets of $\tilde \lambda$ and $\tilde q (\tilde C) = C$. If $int(\supp\cR_{\tilde\alpha,\tilde\beta}) = (x_1,x_2)$ with $x_1 \in [0,1)$,
\[ 0 = a'(x_1)^\vee \times \sigma_{\cR_{\ab,\sigma}}(x_1) =\sum_{n \in \bZ} a'(x_1)^\vee \times \sigma_{\cR_{\tilde\alpha,\tilde\beta}}(x_1+n), \]
so by $q^*a$-positivity of $\cR_{\tilde\alpha,\tilde\beta}$, $x_2 \le x_1+1$, proving the claim.

It remains to show: there does not exist a pair $(\alpha,\beta)$ with $m^l(\cR_{\ab,0}) \le 1$ and $\#(\cM_1(\alpha,\beta)/\bR) = 1$.

\case{1.} If there exists $(\ab)$ with $m^l(\cR_{\ab,0})=0$ and $\#(\cM_1(\ab)/\bR) = 1$, pick one with the minimal $z(\ab):=m(\cR_{\ab,0}) + \sum_{x\in [0,1)} x \cdot m_{\cP_\alpha}(x)/N$. Lemma \ref{lem:nokink} guarantees $0 \le x_1<x_2 <1$ so that $[x_1,x_2] \cap \supp(\cP_\alpha)=\{x_1,x_2\}$ and $\varphi_{a'}(x_1) < \varphi_{a'}(x_2)$. Let $x_0$ be the largest $x\in (x_1,x_2)$ with $\varphi_{a'}(x)=\varphi_{a'}(x_2)$ and let $\alpha'$ be such that $\cR_{\alpha,\alpha,0} = \cT\cT_0$ and $\cR_{\alpha',\alpha,0} = \cT\cR_0$ where $\cT_0$ is a local bigon at $x_2$ and $\cR_0$ is a bigon with $\supp\cR_0 = [x_0,x_2]$ and $I(\cR_0) = 1$.

Suppose $\sigma^1, \sigma^2$ and $\gamma$ satisfy $\#(\cM_1(\alpha',\gamma,\sigma^1)/\bR) = \#(\cM_1(\gamma,\beta,\sigma^2)/\bR) = 1$. Applying Proposition \ref{prop:t3}(c) on the underlying $\bcR^1$ and $\bcR^2$ of $\cR^1:=(\cP_{\alpha'},\cP_\gamma,\sigma^1)$ and $\cR^2:=(\cP_\gamma,\cP_\beta, \sigma^2)$, $\sum m^l(\bcR^j_0) +  \sum \bar I(\bcR^j_0) = 2$ where $\bcR^j_0$ is the (unique by \eqref{eq:fredind}) non-local factor of $\bcR^j$. If neither $\bcR^j$ lifts, $m(\cP_\gamma) = m^s(\bcR^1,\bcR^2) \le 1$, contradicting Lemma \ref{lem:nokink}. Hence, $\cR^2$ does not lift with $m^l(\cR^2) = 0$, while $\cR^1_0$ lifts to $(\cP^0_\bR,\cP^1_\bR)$ with $m(\cP^1_\bR)= m^s(\bcR^1,\bcR^2)=1$. By our choice of $(\ab)$, $\cR^1_0$ must be a bigon due to the first term in $z(\ab)$, and $\gamma = \alpha$ due to the second, contradicting $\<\partial^2\alpha',\beta\> = 0$. Here, $\alpha,\beta$ and $\alpha'$ are admissible by $m^l(\cR_{\ab,0})=0$ and construction.

\case{2.} If $m^l(\cR_{\ab,0})=1$, we take any non-local $\cR_{\ab,\sigma}$ and show $\cM_1(\ab,\sigma) = \emptyset$ for $(\lambda,J)$ sufficiently close to $(\pi_{S^1}^*a,\bar J)$. Then this holds true for any good $\lambda$ and generic $\lambda$-admissible $J$ by automatic transversality \cite{wendl} since $m^h(\cP_\alpha\cup \cP_\beta) = 1$. We proceed similarly to the proof of Proposition \ref{prop:diff-mb} and only highlight the differences. Suppose a sequence $C_n \in \cM^{J_n}_1(\alpha,\beta,\sigma)$ of $J_{r_n}$-holomorphic curves with $(\lambda_n,J_n)\to (\pi_{S^1}^*a,\bar J)$ converges to a $\bar J$-holomorphic building $\bar C = (\bar C^1, \cdots, \bar C^l)$ where $\bar C^i \in \cM(\alpha^i,\beta^i,\sigma^i)$ for orbit sets $\alpha^i$ and $\beta^i$ of $\pi_{S^1}^*a$. Let $\bcP^{i-1}=\bcP_{\alpha^i}$ and $\bcP^i=\bcP_{\beta^i}$ for $1 \le i \le l$ and let $i_1$ and $i_2$ be the smallest and the largest $i$ such that $\bcR^i :=(\bcP^{i-1},\bcP^i,\sigma^i)$ is non-local.

By Lemma \ref{lem:thetaconstraint}, $i_1 < i_2$ and by Proposition \ref{prop:t3}(c) on $(\bcR')^1 = (\bcP^0,\bcP^{i_1},\sum_{i=1}^{i_1}\sigma^i)$ and $(\bcR')^2 = (\bcP^{i_1},\bcP^l, \sum_{i=i_1+1}^l \sigma^i)$,$\sum m^l((\bcR')^j_k) + \sum \bar I((\bcR')^j_k) = 2$. Since at least one $(\bcR')^j$ lifts (see Case 1), $m^s((\bcR')^1,(\bcR')^2) = 1$ and there are exactly two non-local factors, each with $\bar I((\bcR')^j_k)=0$, one of which lifts with $m^l((\bcR')^j_k)=2$ and the other does not lift with $m^l((\bcR')^j_k)=0$. In turn, $\bcR^{i_1}=(\bcR')^1$, $\bcR^{i_2}=(\bcR')^2$ and all other $\bcR^i$ is local. Write $m^s_{\bcR^{i_1},\bcR^{i_2}} = \delta_{x_0}$ and suppose $\brho_{x_0}$ is convex, i.e., $m^l(\bcR^{i_2}_0) = 2$ for the non-local factor of $\bcR^{i_2}$ and $m^l(\bcR^{i_1})=0$. By Step 2 in the original proof, the non-trivial component $\bar C^{i_2}_0$ of $\bar C^{i_2}$ has a positive end at $\brho_{x_0}(\theta^+)$ and no other end at $x_0$ while $\bar C^{i_1}$ has a negative end at $\brho_{x_0}(\theta^-)$ and all other ends at some $\brho_x^m(\theta_x^{\max})$. By Lemma \ref{lem:thetaconstraint},
\[ \Theta(\bar C^{i_1}) = (m(\bcR^{i_1})-1)\theta_h/N - \theta^- = 0  \]
and
\[ \Theta(\bar C^{i_2}_0) = (m(\bcR^{i_2}_0)-2)\theta_h/N + \theta_h + \theta^+ = 0. \]
Hence, $0 < \theta^- < \theta_h$, $-2\theta_h < \theta^+ < -\theta_h$ and there is no flow of $f_{x_0}$ from $\theta^-$ to $\theta^+$. Draw a similar contradiction if $\brho_{x_0}$ is concave.
\end{proof}

\subsection{Toric contact $L(p,q), p \neq 0$}\label{section:lens}

Let $a:\bI \to \ft^2$ be an orbital moment map and suppose $a(i)^\vee \sim (-1)^iu_i$ for a primitive $u_i \in \Lambda$ for $i=0,1$. Collapse $u_i$-orbits at each $\{i\} \times T^2$ to obtain a contact lens space $(Y,\blambda)$ \cite{contactcut}. If $(u_0|u_1) \sim \smat{p & 0 \\ q & 1}$ up to $SL_2(\bZ)$, $Y$ is diffeomorphic to $L(p,q)$ with $H_1(Y)=\Lambda/\Span{u_i}$. Fix $v_i \in \bZ^2$ so that $\det\pmat{u_i|v_i} = (-1)^i$. Over each new embedded orbit $e_i$ with image $(\{i\} \times T^2)/u_i$, $v_i$-action trivializes $\xi$, with respect to which $e_i$ is elliptic with rotation angle $\phi_i$ given by $a'(i)^\vee \sim v_i - \phi_i u_i$. Trivialize $\xi$ over orbits in $Y^o=int(\bI) \times T^2$ as before.

To perturb $\blambda$, let $v_a(x)$ be as in \S \ref{section:mb} on $int(\bI)$ and $v_a(i)=v_i$ for $i=0,1$. Define $\cA_a$, $\Xi_L$ and $N$ as before and choose disjoint neighbourhoods $U_x$ of $x \in \Xi_L$ on which $a' \times a''$ does not vanish. We take a small perturbation $\lambda$ of $\blambda$ which is good on each $U_x$ for $x \in \Xi_L \setminus \{0,1\}$ and unperturbed elsewhere. Assume $m\phi_i \not \in \bZ$ for $0<m<N$ and fix a generic $\lambda$-admissible almost complex structure $J$ on $\bR \times Y$.

Additionally, fix $\epsilon_0,\epsilon_1 > 0$, $\tilde x_0 < -\epsilon_0$ and $\tilde x_1 > 1+\epsilon_1$ and let $V_0:= [\tilde x_0,-\epsilon_0)$, $V_1:=(1+\epsilon_1,\tilde x_1]$, $V_0':=[\tilde x_0,0]$ and $V_1':=[1,\tilde x_1]$. Extend $a$ to an orbital moment map $\tilde a$ on $[\tilde x_0,\tilde x_1]$ so that, for each $i=0,1$, (i) $\tilde a'|_{V_i' \setminus V_i}$ does not annihilate any $nv_i - n'u_i \in \Span_\bZ\{v_i,u_i\}$ with $\abs{n} < N$; (ii) $\tilde a$ is convex on $V_i$; and (iii) $a(i) \times \tilde a'(x) \ge 0$ for $x \in V_i$ with equality only at $x=\tilde x_i$. In this section, paths will be functions on $[\tilde x_0,\tilde x_1]$.

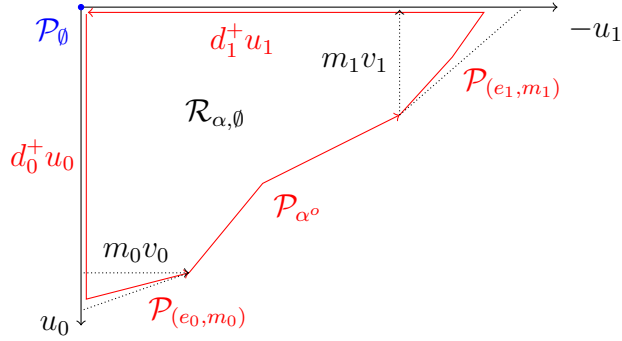
\begin{figure}
\begin{center}
\begin{tikzpicture}[scale=.7]
\draw[->] (-.05,0) -- (-.05,-6) node[left] {$u_0$};
\draw[->] (0,.05) -- (9,.05) node[below right] {$-u_1$};
\draw[color=red,<-] (6,-2) -- (3.4, -3.3) node[below right] {$\cP_{\alpha^o}$} -- (2,-5);
\draw[color=red,<-] (0.08,-.05) -- (7.6,-.05) node[midway, below left] {$d_1^+u_1$}-- (7,-.9) node[below right] {$\cP_{(e_1,m_1)}$} -- (6,-2);
\draw[color=red,<-] (2,-5)--(0.05,-5.5) node[midway, below right] {$\cP_{(e_0,m_0)}$} -- (0.05,-.08) node[midway,left] {$d_0^+u_0$};
\draw[densely dotted,->] (6,-2) -- (6,0) node[midway,left] {$m_1v_1$};
\draw[densely dotted] (6,-2) -- (8.3,0);
\draw[densely dotted,->] (0,-5) -- (2,-5) node[midway,above] {$m_0v_0$};
\draw[densely dotted] (0,-5.7) -- (2,-5);
\draw (2.5,-2) node {$\cR_{\alpha,\emptyset}$};
\draw[fill,blue] (-.05,.05) circle(.05) node[below left] {$\cP_\emptyset$};
\end{tikzpicture}
\end{center}
\vspace{-.3in}
\caption{$\cR_{\alpha,\emptyset}$ for orbit sets in $(S^3,\lambda)$.}
\label{fig:lenspath}
\end{figure}

\begin{definition}\label{def:lensassoc}
Let $\alpha$ and $\beta$ be orbit sets of $\lambda$ with $[\alpha]=[\beta] \in H_1(Y) \isom \bZ/p$.
\begin{enumerate}[(a)]
\item For $i=0,1$, let $w_{i,n}:= nv_i -\floor{n\phi_i}u_i$. To $\gamma =(e_i,m_i)$, we \defn{associate} the unique $\tilde a$-compatible path $\cP_\gamma$ with $m^h_{\cP_\gamma}\equiv 0$ and $m^e_{\cP_\gamma} := \sum_n b_n\delta_{x_n}$, where $n$ appears in $p_{\phi_i}^+(m_i)$ with multiplicity $b_n$ and $x_n$ is the unique $x \in V_i$ with $\tilde a'(x)^\vee \sim w_{i,n}$. To an orbit set $\gamma^o$ in $Y^o$, \defn{associate} $\cP_{\gamma^o}$ as before. In general, we \defn{associate} to $\gamma= \gamma^o \cup \{(e_0,m_0),(e_1,m_1)\}$ the path $\cP_\gamma := \cP_{(e_0,m_0)}\cP_{\gamma^o}\cP_{(e_1,m_1)}$.
\item Write $[\cP_\beta]-[\cP_\alpha] = d_0u_0 + d_1 u_1$ and $d_i = d_i^+-d_i^-$ such that $d_i^\pm \ge 0$ and $d_i^+d_i^-=0$ for each $i$. To $\alpha$ and $\beta$, we \defn{associate} the region $\cR_\ab=(\cP^+_0 \cP_\alpha \cP^+_1,\cP^-_0\cP_\beta\cP^-_1)$ where $\cP^\pm_i$ is $(u_i, 1, d^\pm_i, 0)$ at $\tilde x_i$ and vanishes elsewhere.
\end{enumerate}
\end{definition}

Note $\cP_{(e_i,m_i)}$ is an interpretation of $\Lambda_{\phi_i}^+(m_i)$ in \S \ref{section:ech} as a path. Figure \ref{fig:lenspath} shows $\cR_{\alpha,\emptyset}=(\cP^+,\cP_\emptyset)$ for an orbit set $\alpha$ of $(S^3,\lambda)$, where $\cP^+$ goes around clockwise.

\begin{thm}\label{thm:lens}
Let $(Y,\lambda,J)$ as above. For admissible orbit sets $\alpha$ and $\beta$ of $\lambda$, $\<\partial\alpha,\beta\> =1$ if and only if $\cR_\ab = \cT_1\cR' \cT_2$ where $\cT_i$ are trivial and $\cR'$ is non-local, indecomposable, $\tilde a$-positive, minimally positive and almost minimally decorated.
\end{thm}

We observe the following property of $\cR_\ab$ on $V_i$:
\begin{lem}\label{lem:strata}
If the function $(\tilde a')^\vee \times \sigma_{\cR_\ab}$ vanishes at $x_i \in V_i'$ for $i=0$ (or 1), then $x_i \in V_i$ and $\sigma_{\cR_\ab} =0$ for $x < x_0$ (or $x > x_1$). In particular, if it is positive at $x=i$, then it is non-negative on $V_i'$, $\cR_\ab$ is not minimally decorated and $\sum_{x \in V_i} I_{\cR_\ab}(x) > 0$.
\end{lem}
\begin{proof}
The first statement follows from the definition of $\tilde a$ and $\Lambda_{\phi_i}^+$. The rest follows from this and that all edges of $\cR_\ab$ are elliptic convex.
\end{proof}

\begin{prop}\label{prop:lens}
To adapt previous results, write $\alpha=\alpha^o\cup \{(e_0,m_0^+),(e_1,m_1^+)\}$, $\beta = \beta^o\cup \{(e_0,m_0^-),(e_1,m_1^-)\}$ and $(-1)^i\sigma_{\cR_\ab}(i) = (m_i^--m_i^+)v_i -c_i u_i$.
\begin{enumerate}[(a)]
\item Proposition \ref{prop:index} still holds and $c_\tau(Z)=c_0+c_1$.
\item Lemma \ref{lem:apos} holds after replacing $a$ with $\tilde a$ in the statement.
\item Proposition \ref{prop:indecomposable} now asserts: $\alpha'$ and $\beta'$ do not share orbits; $C'$ has at most one end at covers of $e_0$ or $e_1$; $g(C')=0$; $\cR_\ab = \cT_1 \cR' \cT_2$ where $\cR'$ is indecomposable ($\cR'$ may differ from $\cR_{\alpha',\beta'}$) and $\cT_i$ are trivial; and $\cM_1(\alpha,\beta) \isom \cM_1(\alpha',\beta')$.
\item We say $a:\bI \to (\ft^2)^*$ is \defn{$\cR$-adapted} if it admits an $\cR$-adapted extension $\tilde a$. Proposition \ref{prop:deform} holds with this definition.
\end{enumerate}
\end{prop}

\begin{proof}
(a) Let $\pi$ and $\psi$ be as in the original proof. Define $G^o$ by \eqref{eq:graph} for $\alpha^o$ and $\beta^o$,
\[  G_0 := \{(s, \epsilon(1-\abs{s}))\}_{s \in [-1,1]}, \quad G_1 := \{(s, 1- \epsilon(1-\abs{s}))\}_{s \in [-1,1]}, \]
and let $V$ be the set of multivalent vertices of $G_0 \cup G^o \cup G_1$. Since $H_2(Y) = 0$, we compute $I(\ab,Z)$ on a smooth surface $S$ as in \S \ref{section:ech} subject to: (i) $\pi(S) \subset G_0 \cup G^o \cup G_1 \cup B_{\epsilon/2}(V)$; (ii) for each component $E$ of $G \setminus B_{\epsilon/2}(V)$, $\pi_S^{-1}(E)$ consists of $\abs{c_i}$ disjoint embedded $u_i$-invariant disks if $(i,0) \in E$ for $i=0,1$ and as before, otherwise. We can construct such an $S$ by gluing these pieces as before.

We are ready to compute $I(\ab,[S])$. It suffices to compare $\sum_{x \in V_i} I_{\cR_\ab}(x)$ to
\[  I_i:=\#(\zeta^{-1}(0) \cap S_i) + \#(S_i \cap S'_i) + CZ^I((e_i,m_i^+)) - CZ^I((e_i,m_i^-)) \]
where $\zeta :=x(1-x)\partial_x \in \Gamma(\xi)$ and $S'_i := (\psi\times \id_Y)(S_i)$. The first two terms are equal to $c_i$ and $c_i(m_i^++m_i^-)$, respectively, while $CZ^I((e_i,m_i^\pm)) - \floor{m_i^\pm\phi_i} - m(\cP_{(e_i,m_i^\pm)})$ is twice the area $A_i^\pm$ under the graph of $\Lambda_{\phi_i}^+(m_i^\pm)$ by Pick's theorem. Using these and $d_i = c_i+\floor{m_i^+\phi_i}-\floor{m_i^-\phi_i}$, we get
\[ I_i = (2A_i^+-2A_i^- + c_i(m_i^++m_i^-)) + (m(\cP_{(e_i,m_i^+)}) - m(\cP_{(e_i,m_i^-)}) + d_i). \]
The first summand equals $\sum_{x \in V_i} Q_{\cR_\ab}(x)$ and the second equals $\sum_{x \in V_i} CZ_{\cR_\ab}(x)$.

\vspace{.1in}\noindent(b) Argue as before on $\bI$ and use Lemma \ref{lem:strata} on $V_i'$.

\vspace{.1in}\noindent(c) By Lemma \ref{lem:apos} and by symmetry, it suffices to assume $\sigma_{\cR_\ab}(0)\neq 0$. By Lemma \ref{lem:strata}, $\sigma_{\cR_\ab}(1)=0$. If $C$ is irreducible,
\[ \sum CZ_\tau(e_0^{n_i^+}) = \sum (2\floor{n_i^+\phi_0}+1) = p^+ + 2\floor{m_0^+\phi_0} \]
and
\[ \sum CZ_\tau(e_0^{n_i^-}) = \sum (2\ceil{n_i^-\phi_-}-1) = -p^- + 2\ceil{m_0^-\phi_0} \]
where each sum is over the entries of the partition $(n_1^\pm, \cdots, n_{p^\pm}^\pm)$ of $m_0^\pm$ given by $C$. Substituting these into \eqref{eq:fredind} and using $d_0 = c_0 +\floor{m_0^+\phi_0}-\floor{m_0^-\phi_0}$,
\[  1 =\ind(C) =  2(g(C)-1 + d_0 + p^+ + \max\{0,p^--1\}) + \sum (\pm CZ(\rho_j^\pm) + 1)\]
where the sum is over positive/negative ends of $C$ at $\rho_j^\pm$ in $Y^o$. Thus, if $p^\pm > 0$, $p^+=p^-=1$ and $d_0=0$. By a simple fact for special partitions (\S \ref{section:ech}), $p_{\phi_0}^- = (m_0^-)$ implies $1 \in p_{\phi_0}^+(m_0^-)$, i.e. $m_{\cP_\beta}(x_0) >0$ where $v_a(x_0)= v_0-\floor{\phi_0}u_0$. Since $\cR_\ab$ is positive with $d_0=0$, $m_{\cP_\alpha}(x_0) > 0$ and by index, $m_{\cP_\alpha}|_{V_0} = \delta_{x_0}$. Then, by convexity of $\tilde a$ on $V_0$, $\cR_\ab$ violates $\tilde a$-positivity at $0$. Therefore, $p^+p^-=0$. Moreover, $d_0 + p^+ > 0$ by $\tilde a$-positivity, so $d_0+p^+ = 1$, $p^-\in \{0,1\}$, $g(C)=0$ and $\sum_{x\in \bI}m^l_{\cR_\ab}(x)= 1$. By the last condition, $\cR_\ab$ cannot decompose at $x \in \bI$ and $\alpha$ and $\beta$ do not share orbits in $Y^o$ either. The rest follows easily from this and Lemma \ref{lem:strata}.

\vspace{.1in}\noindent(d) Assume $\sigma_\cR(1)=0$ as above. Let $\{\tilde a_r\}_{r\in [0,1]}$ be a path given by Lemma \ref{lem:lambdaconnected} except: if $m_0^+ + m_0^- > 0$ (so $\supp m_\cR|_{V_0}=\{x_1<\cdots < x_{n_0}\} \neq \emptyset$ and $x_{n_0} > \tilde x_0$), we additionally require $\bar b$ to be smooth on $[x_{n_0},z_{n_0})$ for $z_{n_0} > 0$, which is possible since $a(0) \times v_\cR(x_{n_0}) > 0$. Let $\blambda_r$ be the contact form on $Y$ obtained for each $a_r := \tilde a_r |_\bI$ as above and $\phi_0^r$ the return angle of the orbit of $\blambda_r$ at $x=0$. 

Since $\Lambda_\phi^+(m) = \Lambda_{\phi'}^+(m)$ implies $\Lambda_\phi^+(m')=\Lambda_{\phi'}^+(m')$ for all $m' \le m$, we may assume that $\alpha$ and $\beta$ do not share orbits by (c) and that, for $0 < m \le \max\{m_0^+,m_0^-\}$, $m\phi_0^r$ never crosses an integer during the deformation by the above requirement. This guarantees orbit sets $\alpha_r$ and $\beta_r$ of $\lambda_r$ with $\cR_{\alpha_r,\beta_r} = \cR$ as well as non-degeneracy of the orbits involved. We can carry out the rest of the original proof with minor adjustments.
\end{proof}

\begin{proof}[Proof of Theorem \ref{thm:lens}]
By Lemma \ref{lem:strata}, Proposition \ref{prop:lens} and \S\ref{section:main}, it suffices to show $\<\partial\alpha,\beta\> =1$ when $I(\cR_\ab) = 1$, $\cR_\ab$ is indecomposable, $\sigma_{\cR_\ab}(0) \neq 0$ and $\alpha$ and $\beta$ do not share $e_1$. We can also reduce $\sum_{x \in \bI}m_{\cR_\ab}(x)$ to 1 by induction as in the proof of Proposition \ref{prop:suff} and assume $m_{\cP_\beta}|_\bI \equiv 0$ by duality as in Proposition \ref{prop:base}(iii). Hence, $C \in \cM_1(\ab)$ has one positive end at $\check h_{x^o}$ for $x^o \in int(\bI)$ and the only other end is: (i) none, (ii) a negative end at $e_0^{m_0^-}$, or (iii) a positive end at $e_0^{m_0^+}$. Assume $u_0=(1,0)$ and $1<\phi_0<2$ by $SL_2(\bZ)$-symmetry.

If $\abs{m_0^\pm} \le 1$, Proposition \ref{prop:lens}(d) allows us to deform $\blambda$ to the pullback of $\alpha$ under a suitable identification of $Y' := \pi^{-1}_\bI([0,1-\epsilon))$ (diffeomorphic to $D^2 \times S^1$) with a subset of $(S^2 \times S^1,\alpha)$ in $I_C=\aleph_C+1$ case of \cite[Theorem A.1]{taubes2}. In each of our three cases, a Morse-Bott argument \cite{mb} gives a unique member of $\cM_1(\ab)/\bR$ from an $(\bR \times S^1)$-family of $J$-holomorphic curves there: use (a1) for case (i); (a2) with $p=1,p'=2$ for (ii); and (a3) with $p=p'=1$ for (iii). 

Otherwise, define $\pi:\bR \times Y' \to S^1$ and $q:\bR \times Y' \to \bR \times Y'$ by $\pi(s,x,t_1,t_2)=t_2$ and $q(s,x,t_1,t_2)=(s,x,t_1,mt_2)$. A $J$-holomorphic cylinder $u:\bR \times S^1 \to \bR \times Y'$ with $\deg(\pi\circ u) = \pm m$ lifts (in $m$ different ways) to a $\tilde J$-holomorphic cylinder $\tilde u$ where $q_*\tilde J = J q_*$. Here, $\tilde J$ is $\tilde\lambda$-admissible for a perturbation of $\blambda$ using $f_{x^o}$ with $2m$ critical points but we can pick one $\tilde u$ (with an end at $\brho_{x^o}(\theta_0)$ for a local minimum $\theta_0$) and deform away any other local minimum $\theta$ of $f_{x^o}$ since $\tilde u$ stays away from $\bR \times \brho_{x^o}(\theta)$. By $T^2$-action and Proposition \ref{prop:lens}(d), we reduce the above case.
\end{proof}

\subsection{Toric contact $S^1 \times S^2$}

The discussions from \S \ref{section:lens} work here except, to account for $u_0 = \pm u_1$:
\begin{definition}
The region $\cR_{\ab,d}$ \defn{associated} to $\ab$ and $d \in \bZ$ is the pair $(\cP_0^+\cP_\alpha\cP_1^+,\cP_0^-\cP_\beta\cP_1^-)$ as in Definition \ref{def:lensassoc} except we impose $d_0 = d$. (Note $d_0$ and $d_1$ are not uniquely determined otherwise.)
\end{definition}

\begin{thm}\label{thm:lens0q}
Define $(\lambda,J)$ on $Y$ as in \S \ref{section:lens}. For admissible orbit sets $\alpha$ and $\beta$ of $\lambda$, $\<\partial\alpha,\beta\> =1$ if and only if there exists a unique $d$ such that $\cR_{\ab,d} = \cT_1\cR' \cT_2$ where $\cT_i$ are trivial and $\cR'$ is non-local, indecomposable, $\tilde a$-positive, minimally positive and almost minimally decorated.
\end{thm}

\begin{remark}\label{rem:lens0q}(Uniqueness of $d$)
If $\cR_{\ab,d}$ is minimally positive, $d=0$ or $1$. Suppose both $\cR_{\ab,0}$ and $\cR_{\ab,1}$ satisfy the criteria in Theorem \ref{thm:lens0q} and write $\alpha = \alpha' \cup \gamma$ and $\beta = \beta' \cup \gamma$. By an analogue of Lemma \ref{lem:strata}, $int(\supp\cR_{\ab,1}) = (\tilde x_0,x^o)$ and $int(\supp\cR_{\ab,0})=(x^o,\tilde x_1)$ where $m^l_{\cR_{\alpha',\beta',0}}= \delta_{x_0}$ and as in Remark \ref{rem:uniquelift}, $\supp(m_{\cP_\alpha\cup\cP_\beta}) = \{x^o\}$. Therefore, $(a')^\vee \times u_0$ vanishes precisely at $x^o$, $a'$ is convex at $x^o$ (by $a \times a' >0$), $\alpha' = \check h_{x^o}, \beta' = \emptyset$ and $\gamma$ consists of orbits at $x^o$.
\end{remark}

\begin{proof}[Proof of Theorem \ref{thm:lens0q}]
By Lemma \ref{lem:strata}, we may re-use arguments from \S \ref{section:lens}. In case of non-unique $d$, the two non-zero contributions to $\<\partial\alpha,\beta\>$ cancel.
\end{proof}

\subsection{Map from $ECC(L(p,q))$ to $ECC(T^3)$}

Consider an orbital moment map $a_T:\bR/2\bZ \to (\ft^2)^*$ and suppose there are $\tilde x_0 \in (-1/2,0)$ and $\tilde x_1 \in (1,3/2)$ so that $a_L:=a_T|_\bI$ and $\tilde a_L:=a_T|_{[\tilde x_0,\tilde x_1]}$ satisfy the conditions of $a$ and $\tilde a$ in \S \ref{section:lens}. Suppose further that, for each $i=0,1$, $a_L(i) \times a_T'$ is positive on $\bI$ and negative on $(\tilde x_1,\tilde x_0+2)$. (In particular $p \neq 0$.) As in \S \ref{section:t3} and \S \ref{section:lens}, choose a good perturbation $\lambda_T$ of $\pi_{\bR/2\bZ}^* a_T$ on $(\bR/2\bZ) \times T^2$ and $\lambda_L$ of $\blambda_L$ on $L(p,q)$, as well as generic $\lambda_T$ and $\lambda_L$-admissible $J_T$ and $J_L$. Then, for any orbit set $\alpha$ of $\lambda_L$ with $[\alpha]=0$, $\cR_{\alpha,\emptyset}= (\cP^+,0)$ for a unique $\cP^+:[\tilde x_0,\tilde x_1] \to \cV$.
\begin{prop}
Define $\Phi:ECC(L(p,q),\lambda_L,J_L,0) \to ECC(T^3,\lambda_T,J_T,0)$ by $\cR_{\alpha,\emptyset} = (\cP_{\Phi(\alpha)},0)$. Then, $I(\alpha,\emptyset) = I(\Phi(\alpha),\emptyset)$ and $\partial_T \Phi =\Phi \partial_L$.
\end{prop}

\begin{proof}
Let $\cR = \cR_{\Phi(\alpha),\beta,\sigma}$ for any $\beta$ and $\sigma$ with $\#(\cM_1^{J_T}(\Phi(\alpha),\beta,\sigma)/\bR) =1$. Suppose $\sigma_\cR(x) \neq 0$ for $x \in (\tilde x_1,\tilde x_0+2)$. By positivity of $\cR$ at $\tilde x_0+2$ and the condition on $a'_T$ on $(\tilde x_1, \tilde x_0 + 2)$, $\supp(\cR)$ contains $[\tilde x_1,x]$, and similarly $[x,\tilde x_0+2]$. Moreover, since $I(\cR)=1$, $\supp\cR$ contains $V=[1,\tilde x_0+2]$ or $[\tilde x_1,2]$. Then $\cP_\alpha|_V \equiv 0$ contradicting Lemma \ref{lem:nokink}. Hence, $\supp\cR \subset [\tilde x_0,\tilde x_1]$ and $\cP_\beta = \cP_{\tilde x_0}\cP_0\cP_{\beta^o}\cP_1\cP_{\tilde x_1}$ with $\supp\cP_{\tilde x_i}=\{\tilde x_i\}$, $\supp\cP_i \subset int(V_i)$ and $\supp \cP_{\beta^o} \subset \bI$. If $[\cP_i] = m_iv_i - n_iu_i$, then $n_i \le \phi_im_i$ by convexity of $\tilde a_L|_{V_i}$ and $n_i > \phi_im_i-1$ by $I(\cR)=1$. Thus, by Definition \ref{def:lensassoc}, $[\cP_{\Phi(\beta')}]=[\cP_\beta]$ and $(\cP_{\Phi(\beta')},\cP_\beta,0)$ is positive for the orbit set $\beta' := \{(e_0,m_0),(e_1,m_1)\} \cup \beta^o$ of $\lambda_L$. Since $(\cP_{\Phi(\alpha)},\cP_{\Phi(\beta')},\sigma)$ is positive by Lemma \ref{lem:strata} and $\bar I(\cR) = 0$, $\beta = \Phi(\beta')$ by Lemma \ref{lem:complexity}. Finally, if $m(\cP_{\Phi(\alpha)})=2$ and $\cP_\beta = 0$, then $\cP_{\Phi(\alpha)}|_\bI \equiv 0$, contradicting $I(\cR) =1$. Hence, $(\cP_{\Phi(\alpha)},\cP_\beta)$ has a unique relevant lift, namely $\cR_{\alpha\beta'}$, and the result follows from Theorem \ref{thm:t3} and Theorem \ref{thm:lens}.
\end{proof}

If $a_T(0) \sim (1,0), a_T(1) \sim (0,1)$ and $a_T$ is convex everywhere, we get:
\begin{cor}\cite[Conjecture A.3]{beyond}
If $(S^3, \blambda_L)$ is the boundary of a convex toric domain, $\Phi$ as above is a chain map.
\end{cor}

\bibliographystyle{amsplain}
\bibliography{echcomb}

\end{document}